\documentclass[12pt,a4paper]{article}
\usepackage{latexsym,amssymb,amsfonts,amsmath,amsthm,nccmath,float,enumitem,setspace,bm,authblk,appendix}
\usepackage[usenames,dvipsnames]{xcolor}
\usepackage[hidelinks]{hyperref}
\usepackage[margin=2cm]{geometry}

\setlist{topsep=3pt,partopsep=0pt,itemsep=1pt,parsep=0pt}

\hypersetup{
    colorlinks,
    linkcolor={red!80!black},
    citecolor={blue!80!black},
    urlcolor={blue!80!black}
}

\newtheorem{Theorem}{Theorem}[section]

\newtheorem{Corollary}[Theorem]{Corollary}
\newtheorem{Conjecture}[Theorem]{Conjecture}

\newtheorem{Lemma}[Theorem]{Lemma}

\def \leq {\leqslant}
\def \geq {\geqslant}

\let\oldproofname=\proofname
\renewcommand{\proofname}{\rm\bf{\oldproofname}}

\numberwithin{equation}{section}

\allowdisplaybreaks[4]

\begin{document}

\title{Progress towards generalized Nash-Williams' conjecture on $K_4$-decompositions}

\author[a]{Menglong Zhang}
\author[b]{Gennian Ge}
\affil[a]{Institute of Mathematics and Interdisciplinary Sciences, Xidian University, Xi'an, 710126, P.R. China}
\affil[b]{School of Mathematical Sciences, Capital Normal University, Beijing 100048, P.R. China}
\affil[ ]{mlzhang@bjtu.edu.cn; gnge@zju.edu.cn}

\date{}
\maketitle

\footnotetext{Gennian Ge is supported by the National Key Research and Development Program of China under Grant 2020YFA0712100, the National Natural Science Foundation of China under Grant 12231014, and Beijing Scholars Program.}

%%%%%%%%%%%%%%%%%%%%%%%%%%%%%%%%%%%%%%%%%%%%%%%%%%%%%%%%%%%%%%%%%%%%%%%%%%%%%%%%%%
%%%%%%%%%%%%%%%%%%%%%%%%%%%%%%%%%%%%%%%%%%%%%%%%%%%%%%%%%%%%%%%%%%%%%%%%%%%%%%%%%%

\begin{abstract}
A $K_4$-decomposition of a graph is a partition of its edges into $K_4$s. A fractional $K_4$-decomposition is an assignment of a nonnegative weight to each $K_4$ in a graph such that the sum of the weights of the $K_4$s containing any given edge is one. Formulating a nonlinear programming and reducing the number of variables slowly, we prove that every graph on $n$ vertices with minimum degree at least $\frac{31}{33}n$ has a fractional $K_4$-decomposition. This improves a result of Montgomery that the same conclusion holds for graphs with minimum degree at least $\frac{399}{400}n$. Together with a result of Barber, K\"uhn, Lo, and Osthus, this result implies that for all $\varepsilon> 0$, every large enough $K_4$-divisible graph on $n$ vertices with minimum degree at least  $(\frac{31}{33}+\varepsilon)n$ admits a $K_4$-decomposition.
\end{abstract}


%%%%%%%%%%%%%%%%%%%%%%%%%%%%%%%%%%%%%%%%%%%%%%%%%%%%%%%%%%%%%%%%%%%%%%%%%%%%%%%%%%
%%%%%%%%%%%%%%%%%%%%%%%%%%%%%%%%%%%%%%%%%%%%%%%%%%%%%%%%%%%%%%%%%%%%%%%%%%%%%%%%%%

\section{Introduction}

Let $G$ and $F$ be graphs. An {\em $F$-decomposition} of $G$ is a collection $\mathcal F$ of copies of $F$ in $G$ such that every edge of $G$ is contained in exactly one of these copies.  Let $V(G)$ and $E(G)$ be the set of vertices and edges in $G$, respectively. Let $d_G(v)$ be the {\em degree} of $v$ in $G$ and $\gcd(G)=\gcd\{d_G(v):v\in V(G)\}$ be the greatest common divisor of degrees of the vertices in $G$. Let $e(G)=|E(G)|$ be the number of edges in $G$. Note that $\gcd(F)|\gcd(G)$ and $e(F)|e(G)$ are necessary for the existence of an $F$-decomposition of $G$. If $F$ and $G$ satisfy these two divisibility conditions, then we say that $G$ is {\em $F$-divisible}.

The graph decomposition problem is one of central and classical problems in combinatorics. Kirkman \cite{Kirkman} in 1847 showed that $K_n$ which is $K_3$-divisible has a $K_3$-decomposition, where $K_n$ is the complete graph on $n$ vertices. That is, Kirkman determined the necessary and sufficient condition for the existence of a $K_3$-decomposition of $K_n$. Hanani \cite{Hanani} in 1961 determined the necessary and sufficient condition for the existence of a $K_4$-decomposition of $K_n$. Wilson \cite{Wilson76} stated that for all sufficiently large $n$, if the complete graph $K_n$ is $F$-divisible, then it has an $F$-decomposition.

For an arbitrary graph $G$, the problem of determining whether $G$ admits an $F$-decomposition is much more difficult. It was shown by Dor and Tarsi \cite{DT97} that the problem of determining whether a graph $G$ has an $F$-decomposition is NP-complete. Therefore, it is nature to determine whether all sufficiently dense graphs which are $F$-divisible admit an $F$-decomposition. The most famous conjecture in the area is the Nash-Williams' Conjecture from 1970. Let $\delta(G)=\min\{d_G(v):v\in V(G)\}$ be the {\em minimum degree} of a graph $G$. Nash-Williams \cite{Nash-Williams} conjectured that for every sufficiently large $n$, all $K_3$-divisible graphs $G$ on $n$ vertices with $\delta(G)\geq\frac{3n}{4}$ have a $K_3$-decomposition. The first progress towards the Nash-Williams' Conjecture was made by Gustavsson \cite{Gustavsson} in 1991. He showed that for each integer $r\geq3$, there exist two constant $\varepsilon(r)$ and $N(r)$ such that every $K_r$-divisible graph $G$ on $n>N(r)$ vertices satisfying minimum degree $\delta(G)\geq(1-\varepsilon(r))n$ admits a $K_r$-decomposition. Gustavsson \cite{Gustavsson} generalized the Nash-Williams' Conjecture to $K_r$-decompositions for $r\geq3$ as follows.

\begin{Conjecture}\label{conj:Nash-Williams_conj}
For every $r\geq3$, there exists an $N=N(r)$ such that every $K_r$-divisible graph $G$ on $n\geq N$ vertices with $\delta(G)\geq(1-\frac{1}{r+1})n$ has a $K_r$-decomposition.
\end{Conjecture}

Barber, K\"{u}hn, Lo, and Osthus \cite{BKLO16} gave an approach to converting a fractional decomposition to an exact decomposition.
%Their result is improved by Glock, K\"{u}hn, Lo, Montgomery and Osthus \cite{GKLMO19}. Before stating the result in \cite{GKLMO19}, we need the following definitions.
%Let $\binom{G}{F}$ be the set of copies of $F$ in $G$. A {\em fractional $F$-decomposition} of $G$ is a function $w:\binom{G}{F}\rightarrow\mathbb{R}$ such that for all $e\in E(G)$, $$\sum_{F'\in\binom{G}{F}: e\in E(F')}w(F')=1.$$
Let $G$ and $F$ be graphs, and $\mathcal F(G)$ be the set of copies of $F$ in $G$. A {\em fractional $F$-decomposition} of $G$ is a function $w:\mathcal F(G)\rightarrow[0,1]$ such that for all $e\in E(G)$, $$\sum_{F'\in\mathcal F(G): e\in E(F')}w(F')=1.$$
Note that an $F$-decomposition is a fractional $F$-decomposition with image $\{0,1\}$.
Let $\delta^*_F(n)$ be the least $c > 0$ such that any graph $G$ on $n$ vertices with minimum degree $\delta(G) > cn$ admits a fractional $F$-decomposition, and $\delta^*_F=\limsup\limits_{n\rightarrow\infty} \delta^*_F(n)$ be a {\em fractional F-decomposition threshold}.
A {\em proper coloring} of $F$ is a map $c:V(F)\rightarrow C$ satisfying $c(v)\not= c(w)$ whenever $v,w\in V(F)$ are adjacent. The {\em chromatic number} $\chi(F)$ of $F$ is the smallest integer $k$ such that there is a proper coloring of $F$ with $k$ colors.
%The {\em chromatic number} $\chi(F)$ of $F$ is the smallest integer $k$ such that there is a map $c:V(F)\rightarrow\{1,2,\dots,k\}$ satisfying $c(v)\not= c(w)$ whenever $v\in V(F)$ and $w\in V(F)$ are adjacent.
Glock, K\"{u}hn, Lo, Montgomery and Osthus \cite{GKLMO19} stated the following theorem improving the result in \cite{BKLO16}.
Very recently, Delcourt, Henderson, Lesgourgues and Postle \cite{DHLP25} provided a weak version of the following theorem.
\begin{Theorem}\label{thm:frac->exact}{\rm \cite{GKLMO19}}
Let $\varepsilon> 0$ and $F$ be a graph with chromatic number $\chi=\chi(F)\geq3$. There exists a constant $N$ such that every $F$-divisible graph $G$ on $n>N$ vertices with minimum degree
$$\delta(G)\geq\left(\max\left\{\delta^*_{K_\chi},1-\frac{1}{\chi+1}\right\}+\varepsilon\right)n$$
admits an $F$-decomposition.
\end{Theorem}

For a graph $F$ with $\chi=\chi(F)$, Theorem \ref{thm:frac->exact} implies that determining $\delta^*_{K_\chi}$ is a key ingredient to determine the existence of $F$-decompositions of dense graphs. Yuster \cite{Yuster05} in 2005 showed that $\delta^*_{K_r}\leq1-\frac{1}{9r^{10}}$ for any integer $r\geq3$. In 2012, Dukes \cite{Dukes12} improved this result to $\delta^*_{K_r}\leq1-\frac{2}{9r^2(r-1)^2}$. Barber, K\"uhn, Lo, Montgomery and Osthus \cite{BKLMO17} determined a better bound: $\delta^*_{K_r}\leq1-\frac{1}{10^4r^{3/2}}$. The best current bound on the fractional $K_r$-decomposition threshold for $r\geq4$ is $\delta^*_{K_r}\leq1-\frac{1}{100r}$ by Montgomery \cite{Montgomery19}. For the $K_3$ case, some better bounds are known. In her thesis, Garaschuk \cite{Garaschuk} in 2014 stated that $\delta^*_{K_3}\leq\frac{22}{23}$. Dross \cite{Dross} determined a better bound: $\delta^*_{K_3}\leq 0.9$ by the min-flow max-cut theorem. Dukes and Horsley \cite{DH20} improved the above result to $\delta^*_{K_3}\leq 0.852$ by refining the method used by Dross. In 2021, Delcourt and Postle \cite{DP21} gave the best current bound $\delta^*_{K_3}\leq \frac{7+\sqrt{21}}{14}\approx0.82733$. Very recently, Delcourt, Henderson, Lesgourgues, and Postle \cite{DHLP25_1} disproved Conjecture \ref{conj:Nash-Williams_conj} and showed that there are $c>1$ and infinitely many graphs $G$ on $n$ vertices with minimum degree at least $(1-\frac{1}{c(r+1)})n$ with no $K_r$-decomposition and even no fractional $K_r$-decomposition. This result and Theorem \ref{thm:frac->exact} make determining the correct value of $\delta^*_{K_r}$ a more fascinating problem.

This paper is devoted to improving the fractional $K_4$-decomposition threshold. Recall that the best result of the fractional $K_4$-decomposition threshold is $\delta^*_{K_4}\leq1-\frac{1}{400}$ by Montgomery. We improve this result to $\delta^*_{K_4}\leq1-\frac{2}{33}$ as follows.
%Our main theorem is as follows:
\begin{Theorem}\label{thm:frac_K_4_decom}
If $G$ is a graph on $n$ vertices with minimum degree $\delta(G)\geq\frac{31}{33}n$, then $G$ admits a fractional $K_4$-decomposition.
\end{Theorem}

Combining Theorem \ref{thm:frac->exact} with our result, we have the following progress on Conjecture \ref{conj:Nash-Williams_conj}:
\begin{Theorem}\label{thm:K_4_decom}
Let $\varepsilon> 0$. There is a constant $N$ such that every $K_4$-divisible graph $G$ on $n>N$ vertices with minimum degree $\delta(G)\geq(\frac{31}{33}+\varepsilon)n$ has a $K_4$-decomposition.
\end{Theorem}

%Note that $1-\frac{2}{33}=0.9\dot{3}\dot{9}$ and Conjecture \ref{conj:Nash-Williams_conj} implies the desired lower bound is $0.8$.
In fact, the more general theorem is obtained by Theorems \ref{thm:frac->exact} and \ref{thm:frac_K_4_decom}.
\begin{Theorem}\label{thm:F_decom_ch=4}
Let $\varepsilon> 0$ and $F$ be a graph with chromatic number $\chi(F)=4$. There exists a constant $N$ such that every $F$-divisible graph $G$ on $n>N$ vertices with minimum degree
$\delta(G)\geq(\frac{31}{33}+\varepsilon)n$
admits an $F$-decomposition.
\end{Theorem}

Combined with the work of Condon, Kim, K\"uhn, and Osthus (see Theorem 1.2 in \cite{CKKO19}), Theorem \ref{thm:frac_K_4_decom} gives the following theorem. Here we say that a collection $\mathcal{H} = \{H_1,\dots,H_s \}$ of graphs {\em packs} into $G$ if there exist pairwise edge-disjoint copies
of $H_1,\dots,H_s$ in $G$. Let $\Delta(G)=\max\{d_G(v):v\in V(G)\}$ be the {\em maximum degree} of a graph $G$ and $e(\mathcal H) =\sum_{H\in\mathcal H}e(H)$. A graph $H$ on $n$ vertices is said to be {\em $\eta$-separable} if there exists a set $S\subseteq V(H)$ such that $|S|\leq\eta n$ and the size of every component of $H[V(H)\setminus S]$ is at most $\eta n$. And a graph $H$ on $n$ vertices is {\em $(r,\eta)$-chromatic} if the graph $F'$, obtained from $F$ by deleting isolated vertices, can be properly colored with $r+1$ colors such that the size of some colour class is at most $\eta n$.
\begin{Theorem}\label{thm:separable_packing}
For all $\Delta$, $0<\nu<1$ and $\frac{31}{33}<\delta\leq1$, there exist $\xi,\eta>0$ and $n_0\in\mathbb N$ such that for all $n>n_0$ the following holds. Suppose that $\mathcal H$ is a collection of $n$-vertex $(4,\eta)$-chromatic $\eta$-separable graphs and $G$ is an $n$-vertex graph such that
\begin{enumerate}
\item[$(i)$] $(\delta-\xi)n\leq \delta(G)\leq\Delta(G)\leq(\delta+\xi)n$,
\item[$(ii)$] $\Delta(H)\leq \Delta$ for all $H\in\mathcal H$, and
\item[$(iii)$] $e(\mathcal H)\leq(1-\nu)e(G)$.
\end{enumerate}
Then $\mathcal H$ packs into $G$.
\end{Theorem}

\section{Fractional $K_4$-decompositions}
We open this section with the definition of an edge-gadget introduced by Barber, K\"uhn, Lo, Montgomery, and Osthus \cite{BKLMO17}.
Let $G$ be a graph. Let $\mathcal K_l(G)$ be the set of copies of $K_l$ in $G$ for $l\geq 2$ and $\mathcal K_l(G,S)=\{K\in\mathcal K_l(G):S\subseteq V(K)\}$ for $S\subseteq V(G)$.
%Now we present the definition of an edge-gadget introduced by Barber, K\"uhn, Lo, Montgomery, and Osthus \cite{BKLMO17}, as follows.
The {\em edge-gadget} of $e$ in $K\in\mathcal K_6(G)$ is a function $\psi_{K,e}:\mathcal K_4(G)\rightarrow\mathbb{R}$ with
\[ \psi_{K,e}(T) = \left\{
\begin{array}{cl}
\frac{1}{2}, & \text{if } T\subseteq K \text{ and }e\cap V(T)=\emptyset,\\
-\frac{1}{6}, & \text{if } T\subseteq K \text{ and } |e\cap V(T)|=1,\\
\frac{1}{6}, & \text{if } T\subseteq K \text{ and } e\in E(T),\\
0, & \text{otherwise}.
\end{array} \right. \]
The following lemma is a basic and useful proposition of the edge-gadget.
\begin{Lemma}\label{lem:edge_gadget}
For any fixed $e\in E(G)$ and $K\in\mathcal K_6(G)$, if $f\in E(G)$, then
\[ s(f)=\sum_{T\in\mathcal K_4(G,f)}\psi_{K,e}(T) = \left\{
\begin{array}{cl}
1, & \text{if } f=e, \text{ and }\\
0, & \text{otherwise}.
\end{array} \right. \]
\end{Lemma}
\begin{proof} %Let $s(f)=\sum_{T\in\mathcal K_4(G,f)}\psi_{K,e}(T)$.
If $f\not\in E(K)$, then $\psi_{K,e}(T)=0$ for any $T\in\mathcal K_4(G,f)$, and so $s(f)=0$.
Suppose that $f\in E(K)$. If $|f\cap e|=1$, then $s(f)=3\cdot\frac{1}{6}+\binom{3}{2}\cdot(-\frac{1}{6})+[\binom{n-2}{2}-3-\binom{3}{2}]\cdot0=0$. If $f\cap e=\emptyset$, then $s(f)=\frac{1}{2}+2\cdot2\cdot(-\frac{1}{6})+\frac{1}{6}+[\binom{n-2}{2}-1-2\cdot2-1]\cdot0=0$. If $f=e$, then $s(f)=\binom{4}{2}\cdot\frac{1}{6}+[\binom{n-2}{2}-\binom{4}{2}]\cdot0=1$.
\end{proof}

Let $G$ be a graph. An {\em ordered $l$-clique} of $G$ is an $l$-tuple $(v_1,v_2,\dots,v_l)$ such that $\{v_1,\dots, v_l\}$ is a vertex set of $K_l$ in $G$ and let $\mathcal{OK}_l(G)$ denote the set of ordered $l$-cliques in $G$. If $K = (v_1 , \dots, v_l )\in \mathcal{OK}_l(G)$, then let $V(K) = \{v_1 , \dots, v_l\}$.
Define $N(v)=\{u\in V(G):\{u,v\}\in E(G)\}$ and $N(S)=\bigcap_{v\in S}N(v)$ for any $S\subseteq V(G)$. For convenience, we shall write $N(v_1,\dots,v_r)$ instead of $N(\{v_1,\dots,v_r\})$ for $\{v_1,\dots,v_r\}\subseteq V(G)$ sometime.
%For convenience, throughout this paper, we shall write $N(v_1,\dots,v_r)$ instead of $N(\{v_1,\dots,v_r\})$ for $\{v_1,\dots,v_r\}\subseteq V(G)$.
For any $r\in\{2,3,4,5\}$ and any $K=(v_1,\dots,v_r)\in\mathcal{OK}_r(G)$, let
\begin{align}\label{def:W(K)}
W(v_1,\dots,v_r)=W(K)=\prod_{i=2}^{r}\frac{1}{|N(v_1,\dots,v_i)|}.
\end{align}
%Moreover, let $W(v_1,\dots,v_r)=W(K)$.
Meanwhile, we define a function $W_G:\mathcal K_4(G)\rightarrow\mathbb{R}$ as $$W_G(T)=\frac{1}{2}\cdot\sum_{(v_1,\dots,v_6)\in\mathcal{OK}_6(G)}W(v_1,\dots,v_5)\cdot\psi_{K',e}(T),$$
where $K'=G[\{v_1,\dots,v_6\}]$ is an induced subgraph in $G$ and $e=\{v_1,v_2\}$.

We will prove that $W_G(T)$ is a fractional $K_4$-decomposition. The following theorem implies the weight of each edge is $1$ for our weight $W_G(T)$ of $K_4$s.
\begin{Theorem}\label{thm:frac_dec_equ=1}
Let $G$ be a graph with $n$ vertices and minimum degree $\delta(G)>\frac{4}{5}n$. If $e\in E(G)$, then $$\sum_{T\in\mathcal K_4(G,e)}W_G(T)=1.$$
\end{Theorem}
\begin{proof}
Let $S_i=\{v_1,\dots, v_i\}$ for $i\in\{2,3,4,5\}$. Since $\delta(G)>\frac{4}{5}n$,
\begin{align*}
&|N(S_2)|=|N(v_1)|+|N(v_2)|-|N(v_1)\cup N(v_2)|>2\cdot\frac{4}{5}n-n=\frac{3}{5}n, \notag \\
&|N(S_3)|=|N(S_2)|+|N(v_3)|-|N(S_2)\cup N(v_3)|>\frac{3}{5}n+\frac{4}{5}n-n=\frac{2}{5}n, \notag \\
&|N(S_4)|=|N(S_3)|+|N(v_4)|-|N(S_3)\cup N(v_4)|>\frac{2}{5}n+\frac{4}{5}n-n=\frac{1}{5}n, \text{ and }\notag \\
&|N(S_5)|=|N(S_4)|+|N(v_5)|-|N(S_4)\cup N(v_5)|>\frac{1}{5}n+\frac{4}{5}n-n=0. \notag
\end{align*}
Recall that $W(v_1,v_2,v_3,v_4,v_5)=\frac{1}{|N(S_2)||N(S_3)||N(S_4)||N(S_5)|}$. So $W(v_1,v_2,v_3,v_4,v_5)$ is well-defined and $W(v_1,v_2,v_3,v_4,v_5)>0$. For every $e\in E(G)$, let $W_e=\sum_{T\in\mathcal K_4(G,e)}W_G(T)$. By the definition of $W_G(T)$,
\begin{align*}
W_e&=\sum_{T\in\mathcal K_4(G,e)}W_G(T)=\sum_{T\in\mathcal K_4(G,e)}\frac{1}{2}\cdot\sum_{K=(v_1,\dots,v_6)\in\mathcal{OK}_6(G)}W(v_1,\dots,v_5)\cdot\psi_{G[V(K)],\{v_1,v_2\}}(T)\\
&=\frac{1}{2}\cdot\sum_{K=(v_1,\dots,v_6)\in\mathcal{OK}_6(G)}W(v_1,\dots,v_5)\cdot\sum_{T\in\mathcal K_4(G,e)}\psi_{G[V(K)],\{v_1,v_2\}}(T)
\end{align*}
By Lemma \ref{lem:edge_gadget}, for any fixed $K=(v_1,\dots,v_6)\in\mathcal{OK}_6(G)$,
$$\sum_{T\in\mathcal K_4(G,e)}\psi_{G[V(K)],\{v_1,v_2\}}(T) = \left\{
\begin{array}{cl}
1, & \text{if } e=\{v_1,v_2\}, \text{ and }\\
0, & \text{otherwise}.
\end{array} \right.$$
Then
\begin{align*}
W_e&=\frac{1}{2}\cdot\sum_{\substack{K=(v_1,\dots,v_6)\in\mathcal{OK}_6(G) \\ e=\{v_1,v_2\}}}W(v_1,\dots,v_5)\\
&=\frac{1}{2}\cdot\sum_{\substack{K=(v_1,\dots,v_6)\in\mathcal{OK}_6(G) \\ e=\{v_1,v_2\}}}\prod_{i=2}^{5}\frac{1}{|N(S_i)|}.
\end{align*}
Note that for each $i\in\{2,3,4,5\}$, $|\{(v_1,\dots,v_i,s)\in\mathcal{OK}_{i+1}(G):s\in V(G)\}|=|N(v_1,\dots,v_i)|=|N(S_i)|$. Therefore,
\begin{align*}
W_e&=\frac{1}{2}\cdot\sum_{\substack{K=(v_1,\dots,v_6)\in\mathcal{OK}_6(G) \\ e=\{v_1,v_2\}}}\frac{1}{|N(S_2)||N(S_3)||N(S_4)||N(S_5)|}\\
&=\frac{1}{2}\cdot\sum_{\substack{K=(v_1,\dots,v_5)\in\mathcal{OK}_5(G) \\ e=\{v_1,v_2\}}}\frac{1}{|N(S_2)||N(S_3)||N(S_4)|}\\
&=\frac{1}{2}\cdot\sum_{\substack{K=(v_1,\dots,v_4)\in\mathcal{OK}_4(G) \\ e=\{v_1,v_2\}}}\frac{1}{|N(S_2)||N(S_3)|}\\
&=\frac{1}{2}\cdot\sum_{\substack{K=(v_1,\dots,v_3)\in\mathcal{OK}_3(G) \\ e=\{v_1,v_2\}}}\frac{1}{|N(S_2)|}\\
&=\frac{1}{2}\cdot\sum_{\substack{K=(v_1,v_2)\in\mathcal{OK}_2(G) \\ e=\{v_1,v_2\}}}1\\
&=\frac{1}{2}\cdot 2=1.
\end{align*}
\end{proof}

Theorem \ref{thm:frac_dec_equ=1} implies that if $W_G(T)\geq0$ for any $T\in\mathcal{K}_4(G)$, then $W_G(T)$ is the desired fractional $K_4$-decomposition.

For a subgraph $H$ of $G$,  $\mathcal {OK}_r(G,H)$ denotes the set of elements $K\in \mathcal {OK}_r(G)$ such that $V(H)\subseteq V(K)$.
Let $s\geq r \geq 1$, $H_1= (u_1 ,\dots, u_s )\in\mathcal{OK}_s(G)$ and $H_2= (v_1 ,\dots, v_r)\in\mathcal{OK}_r(G)$. We say $H_1$ is an {\em ordered subgraph} of $H_2$ if $u_1\dots u_r$ is a (not necessarily consecutive) subsequence of $v_1\dots v_s$.
For an ordered $ H\in\mathcal{OK}_r(G)$, $\mathcal{OK}_s(G, H)$ denote the set of elements in $\mathcal{OK}_s(G)$ containing $H$ as an ordered subgraph. We define a function $W_G:\mathcal {OK}_4(G)\rightarrow\mathbb{R}$ as
$$W_G(O)=\frac{1}{2}\sum_{K=(v_1,\dots,v_6)\in\mathcal {OK}_6(G,O)}W(v_1,\dots,v_5)\psi_{G[V(K)],\{v_1,v_2\}}(G[V(O)]).$$

\begin{Lemma}\label{lem:sum_W_G(O)}
For any $T\in\mathcal K_4(G)$, $\sum_{O\in\mathcal {OK}_4(G,T)}W_G(O)=W_G(T)$. Moreover, if $W_G(O)\geq0$ for any $O\in\mathcal{OK}_4(G)$, then $W_G(T)\geq0$ for any $T\in\mathcal{K}_4(G)$.
\end{Lemma}
\begin{proof}
For $T\in\mathcal K_4(G)$, by the definition of $W_G(O)$,
\begin{align*}
\sum_{O\in\mathcal {OK}_4(G,T)}W_G(O)&=\frac{1}{2}\sum_{O\in\mathcal {OK}_4(G,T)}\sum_{K=(v_1,\dots,v_6)\in\mathcal {OK}_6(G,O)}W(v_1,\dots,v_5)\psi_{G[V(K)],\{v_1,v_2\}}(G[V(O)])\\
&=\frac{1}{2}\sum_{K=(v_1,\dots,v_6)\in\mathcal {OK}_6(G,O)}\sum_{O\in\mathcal {OK}_4(G,T)}W(v_1,\dots,v_5)\psi_{G[V(K)],\{v_1,v_2\}}(T)\\
&=\frac{1}{2}\sum_{\substack{K=(v_1,\dots,v_6)\in\mathcal {OK}_6(G)\\ V(T)\subseteq V(K)}}W(v_1,\dots,v_5)\psi_{G[V(K)],\{v_1,v_2\}}(T)\\
&=\frac{1}{2}\sum_{\substack{K=(v_1,\dots,v_6)\in\mathcal {OK}_6(G)}}W(v_1,\dots,v_5)\psi_{G[V(K)],\{v_1,v_2\}}(T)\\
&=W_G(T),
\end{align*}
where the penultimate equality holds since $\psi_{G[V(K)],\{v_1,v_2\}}(T)=0$ for any $K\in\mathcal {OK}_6(G)$ with $V(T)\not\subseteq V(K)$.
Thus if $W_G(O)\geq0$ for any $O\in\mathcal{OK}_4(G)$, then $W_G(T)\geq0$ for any $T\in\mathcal{K}_4(G)$.
\end{proof}
Hence to prove $W_G(T)\geq0$, it suffices to prove $W_G(O)\geq0$ by the above lemma. For each $O=(x_1,x_2,x_3,x_4)\in\mathcal {OK}_4(G)$, let
%We define a function $W'_G:\mathcal {OK}_4(G)\rightarrow\mathbb{R}$ as
\begin{align*}
W'_G(O)=&1-\frac{12}{W(x_1,x_2,x_3)}W_G(O).
%\\
%=&\frac{1}{W(x_1,x_2,x_3)}\sum_{y\in R}[2W(x_1,y,x_2,x_3)-W(x_1,x_2,x_3,y)-W(x_1,x_2,y,x_3)+\\
%&\sum_{z\in N(y)\cap R}(2W(x_1,y,x_2,x_3,z)+2W(x_1,y,x_2,z,x_3)+2W(x_1,y,z,x_2,x_3)-\\
%&W(x_1,x_2,x_3,y,z)-W(x_1,x_2,y,x_3,z)-W(x_1,x_2,y,z,x_3)-3W(y,z,x_1,x_2,x_3))].
\end{align*}
The following lemma implies that in order to prove $W_G(O)\geq 0$, it suffices now to prove $W'_G(O)\leq 1$ and gives a new expression of $W'_G(O)$.
\begin{Lemma}\label{lem:W'_G(O)->W_G(O)}
Let $G$ be a graph with $n$ vertices and minimum degree $\delta(G)>\frac{4}{5}n$. Let $O=(x_1,x_2,x_3,x_4)\in\mathcal {OK}_4(G)$ and $R=N(x_1,x_2,x_3,x_4)$. Then the following hold.
\begin{enumerate}
\item[$(1)$] If $W'_G(O)\leq 1$, then $W_G(O)\geq 0$.
\item[$(2)$] We know that
\begin{align*}
W'_G(O)&=\frac{1}{W(x_1,x_2,x_3)}\sum_{y\in R}[2W(x_1,y,x_2,x_3)-W(x_1,x_2,x_3,y)-W(x_1,x_2,y,x_3)+\\
&\sum_{z\in N(y)\cap R}(2W(x_1,y,x_2,x_3,z)+2W(x_1,y,x_2,z,x_3)+2W(x_1,y,z,x_2,x_3)-\\
&W(x_1,x_2,x_3,y,z)-W(x_1,x_2,y,x_3,z)-W(x_1,x_2,y,z,x_3)-3W(y,z,x_1,x_2,x_3))].
\end{align*}
\end{enumerate}
%Let $O=(x_1,x_2,x_3,x_4)\in\mathcal {OK}_4(G)$ and $R=N(x_1,x_2,x_3,x_4)$. Then
%\begin{align*}
%W_G(O)&=\frac{1}{12}W(x_1,x_2,x_3)+\frac{1}{12}\sum_{y\in R}[W(x_1,x_2,x_3,y)+W(x_1,x_2,y,x_3)-2W(x_1,y,x_2,x_3)+\\
%&\sum_{z\in N(y)\cap R}(W(x_1,x_2,x_3,y,z)+W(x_1,x_2,y,x_3,z)+W(x_1,x_2,y,z,x_3)-2W(x_1,y,x_2,x_3,z)\\
%&-2W(x_1,y,x_2,z,x_3)-2W(x_1,y,z,x_2,x_3)+3W(y,z,x_1,x_2,x_3))].
%\end{align*}
\end{Lemma}
\begin{proof}
(1) Since $\delta(G)>\frac{4}{5}n$,
$$|N(x_1,x_2)|=|N(x_1)|+|N(x_2)|-|N(x_1)\cup N(x_2)|>2\cdot\frac{4}{5}n-n=\frac{3}{5}n,$$
and
$$|N(x_1,x_2,x_3)|=|N(x_1,x_2)|+|N(x_3)|-|N(x_1,x_2)\cup N(x_3)|>\frac{3}{5}n+\frac{4}{5}n-n=\frac{2}{5}n.$$
Then $W(x_1,x_2,x_3)=\frac{1}{|N(x_1,x_2)|\cdot|N(x_1,x_2,x_3)|}>0$ is well-defined. By $W'_G(O)\leq1$, $W_G(O)=\frac{W(x_1,x_2,x_3)}{12}(1-W'_G(O))\geq0$.

%Let $O=(x_1,x_2,x_3,x_4)\in\mathcal {OK}_4(G)$ and $R=N(x_1,x_2,x_3,x_4)$.
(2) By the definition of $W_G(O)$,
\begin{align*}
W_G(O)=&\frac{1}{2}\sum_{K=(v_1,\dots,v_6)\in\mathcal {OK}_6(G,O)}W(v_1,\dots,v_5)\psi_{G[V(K)],\{v_1,v_2\}}(G[V(O)]) \\
=&\frac{1}{2}\sum_{y\in R}\sum_{z\in N(y)\cap R}[\frac{1}{6}(W(x_1,x_2,x_3,x_4,y)+W(x_1,x_2,x_3,y,x_4)+W(x_1,x_2,x_3,y,z)+\\
&W(x_1,x_2,y,x_3,x_4)+W(x_1,x_2,y,x_3,z)+W(x_1,x_2,y,z,x_3))-\frac{1}{6}(W(x_1,y,x_2,x_3,x_4)+\\
&W(x_1,y,x_2,x_3,z)+W(x_1,y,x_2,z,x_3)+W(x_1,y,z,x_2,x_3)+W(y,x_1,x_2,x_3,x_4)+\\
&W(y,x_1,x_2,x_3,z)+W(y,x_1,x_2,z,x_3)+W(y,x_1,z,x_2,x_3))+\frac{1}{2}W(y,z,x_1,x_2,x_3)]\\
=&\frac{1}{2}\sum_{y\in R}\sum_{z\in N(y)\cap R}[\frac{1}{6}(W(x_1,x_2,x_3,x_4,y)+W(x_1,x_2,x_3,y,x_4)+W(x_1,x_2,x_3,y,z)+\\
&W(x_1,x_2,y,x_3,x_4)+W(x_1,x_2,y,x_3,z)+W(x_1,x_2,y,z,x_3))-\frac{1}{3}(W(x_1,y,x_2,x_3,x_4)+\\
&W(x_1,y,x_2,x_3,z)+W(x_1,y,x_2,z,x_3)+W(x_1,y,z,x_2,x_3))+\frac{1}{2}W(y,z,x_1,x_2,x_3)],
\end{align*}
where the last equality holds since $W(x_1,y,x_2,x_3,x_4)=W(y,x_1,x_2,x_3,x_4)$, $W(x_1,y,x_2,x_3,z)=W(y,x_1,x_2,x_3,z)$, $W(x_1,y,x_2,z,x_3)=W(y,x_1,x_2,z,x_3)$ and $W(x_1,y,z,x_2,x_3)=W(y,x_1,z,x_2,x_3)$ by $(\ref{def:W(K)})$.

Note that $W(x_1,x_2,x_3,x_4,y)$, $W(x_1,x_2,x_3,y,x_4)$, $W(x_1,x_2,y,x_3,x_4)$ and $W(x_1,y,x_2,x_3,x_4)$ do not depend on $z$.
Then
$$\sum_{z\in N(y)\cap R}W(x_1,x_2,x_3,x_4,y)=\sum_{z\in N(y)\cap R}W(x_1,x_2,x_3,x_4)\frac{1}{|N(x_1,x_2,x_3,x_4,y)|}=W(x_1,x_2,x_3,x_4),$$
$$\sum_{z\in N(y)\cap R}W(x_1,x_2,x_3,y,x_4)=\sum_{z\in N(y)\cap R}W(x_1,x_2,x_3,y)\frac{1}{|N(x_1,x_2,x_3,y,x_4)|}=W(x_1,x_2,x_3,y),$$
$$\sum_{z\in N(y)\cap R}W(x_1,x_2,y,x_3,x_4)=\sum_{z\in N(y)\cap R}W(x_1,x_2,y,x_3)\frac{1}{|N(x_1,x_2,y,x_3,x_4)|}=W(x_1,x_2,y,x_3),$$
and
$$\sum_{z\in N(y)\cap R}W(x_1,y,x_2,x_3,x_4)=\sum_{z\in N(y)\cap R}W(x_1,y,x_2,x_3)\frac{1}{|N(x_1,y,x_2,x_3,x_4)|}=W(x_1,y,x_2,x_3).$$
\begin{align*}
W_G(O)
&=\frac{1}{2}\sum_{y\in R}\frac{1}{6}(W(x_1,x_2,x_3,x_4)+W(x_1,x_2,x_3,y)+W(x_1,x_2,y,x_3))-\frac{1}{3}W(x_1,y,x_2,x_3)\\
&+\sum_{z\in N(y)\cap R}[\frac{1}{6}(W(x_1,x_2,x_3,y,z)+W(x_1,x_2,y,x_3,z)+W(x_1,x_2,y,z,x_3))\\
&-\frac{1}{3}(W(x_1,y,x_2,x_3,z)+W(x_1,y,x_2,z,x_3)+W(x_1,y,z,x_2,x_3))+\frac{1}{2}W(y,z,x_1,x_2,x_3)].
\end{align*}
Similarly, $W(x_1,x_2,x_3,x_4)$ does not depend on $y$, and so
\begin{align*}
\sum_{y\in R}W(x_1,x_2,x_3,x_4)=\sum_{y\in R}W(x_1,x_2,x_3)\frac{1}{|N(x_1,x_2,x_3,x_4)|}=W(x_1,x_2,x_3).
\end{align*}

Therefore,
\begin{align*}
W_G(O)=&\frac{1}{12}W(x_1,x_2,x_3)+\frac{1}{12}\sum_{y\in R}[W(x_1,x_2,x_3,y)+W(x_1,x_2,y,x_3)-\\
&2W(x_1,y,x_2,x_3)+\sum_{z\in N(y)\cap R}(W(x_1,x_2,x_3,y,z)+W(x_1,x_2,y,x_3,z)\\
&+W(x_1,x_2,y,z,x_3)-2W(x_1,y,x_2,x_3,z)-2W(x_1,y,x_2,z,x_3)-\\
&2W(x_1,y,z,x_2,x_3)+3W(y,z,x_1,x_2,x_3))].
\end{align*}
Then the desired conclusion is obtained.
\end{proof}

\begin{Theorem}\label{thm:W'_G(O)<=1}
If $G$ is a graph on $n$ vertices with minimum degree $\delta(G)\geq(1-\frac{2}{33})n$, then $W'_G(O)\leq 1$ for every $O\in\mathcal {OK}_4(G)$.
\end{Theorem}
Theorem \ref{thm:W'_G(O)<=1} is proved in the following two sections by solving a nonlinear programming. We conclude this section by proving Theorem \ref{thm:frac_K_4_decom} assuming Theorem \ref{thm:W'_G(O)<=1} holds.
\begin{proof}[{\bf Proof of Theorem \ref{thm:frac_K_4_decom}}]
Suppose that Theorem \ref{thm:W'_G(O)<=1} holds. That is, $W'_G(O)\leq 1$ for each $O\in\mathcal {OK}_4(G)$. it follows from Lemma \ref{lem:W'_G(O)->W_G(O)} (1) that $W_G(O)\geq 0$ for every $O\in\mathcal {OK}_4(G)$. Then $W_G(T)\geq 0$ for all $T\in\mathcal {K}_4(G)$ by Lemma \ref{lem:sum_W_G(O)}.

Since $\frac{31}{33}>\frac{4}{5}$, Theorem \ref{thm:frac_dec_equ=1} implies that $\sum_{T\in\mathcal K_4(G,e)}W_G(T)=1$ for any $e\in E(G)$.

%Since $\frac{31}{33}>\frac{4}{5}$, it follows from Theorem \ref{thm:frac_dec_equ=1} that $\sum_{T\in\mathcal K_4(G,e)}W_G(T)=1$ for any $e\in E(G)$.

Thus $W_G(T)$ is the desired fractional $K_4$-decomposition.
\end{proof}

\section{Formulate a nonlinear programming}
%为了将后面的变量都变成有界的即均在[0,1]中
To obtain a nonlinear programming that each variable is in $[0,1]$, we introduce the following definition and provide some useful lemmas.
\subsection{The common neighbor density}
Let $S\subseteq V(G)$ and $|V(G)|=n$. The {\em common neighbor density} of $S$ is defined as
$$\widehat{N}(S)=\frac{|V(G)\setminus(\bigcup_{s\in S}(V(G)\setminus N(s)))|}{n}.$$
Similarly, we use $\widehat{N}(v_1,\dots,v_r)$ instead of $\widehat{N}(S)$ for $S=\{v_1,\dots,v_r\}\subseteq V(G)$ sometime.
Note that $\widehat{N}(\emptyset)=1$ and if $S\neq\emptyset$, then $\widehat{N}(S)=\frac{|\bigcap_{s\in S}N(s)|}{n}=\frac{|N(S)|}{n}$ by De Morgan's laws.

The following two lemmas about the common neighbor density are mentioned in \cite{DP21}. For completeness, we also provide a proof here.
\begin{Lemma}\label{lem:decrease_N(K)}
Let $G$ be a graph. If $S\subseteq S'\subseteq V(G)$, then $\widehat{N}(S)\geq\widehat{N}(S')$.
\end{Lemma}
\begin{proof}
If $S=\emptyset$, then $\widehat{N}(S)=1$, and so $\widehat{N}(S')\leq 1=\widehat{N}(S)$. Suppose that $\emptyset \neq S\subseteq S'$. Then $N(S')=\bigcap_{s\in S'}N(s)\subseteq\bigcap_{s\in S}N(s)=N(S)$, and so $\widehat{N}(S)\geq\widehat{N}(S')$.
%$V(G)\setminus(\bigcup_{s\in S'}(V(G)\setminus N(s)))\subseteq V(G)\setminus(\bigcup_{s\in S}(V(G)\setminus N(s)))$. Thus, $\widehat{N}(S)\geq\widehat{N}(S')$.
\end{proof}
\begin{Lemma}\label{lem:N(S)_1}
Let $G$ be a graph. If $A,B\subseteq V(G)$, then $\widehat{N}(A\cup B)\geq \widehat{N}(A)+\widehat{N}(B)-\widehat{N}(A\cap B)$.
\end{Lemma}
\begin{proof}
If $A\cup B=\emptyset$, then $A=B=\emptyset$ and so $\widehat{N}(A\cup B)=\widehat{N}(A)=\widehat{N}(B)=\widehat{N}(A\cap B)=1$. Thus, $\widehat{N}(A\cup B)\geq\widehat{N}(A)+\widehat{N}(B)-\widehat{N}(A\cap B)$.

Suppose that $A=\emptyset$ and $B\neq\emptyset$. Then $\widehat{N}(A\cup B)=\widehat{N}(B)$ and $\widehat{N}(A)=\widehat{N}(A\cap B)=1$. Therefore, $\widehat{N}(A\cup B)\geq\widehat{N}(A)+\widehat{N}(B)-\widehat{N}(A\cap B)$. Similarly, if $A\neq\emptyset$ and $B=\emptyset$, then $\widehat{N}(A\cup B)=\widehat{N}(A)$ and $\widehat{N}(B)=\widehat{N}(A\cap B)=1$. Thus $\widehat{N}(A\cup B)\geq\widehat{N}(A)+\widehat{N}(B)-\widehat{N}(A\cap B)$.

Suppose that $A\neq\emptyset$ and $B\neq\emptyset$. Recall that $N(A\cup B)=N(A)\cap N(B)$. Then $$|N(A\cup B)|=|N(A)\cap N(B)|=|N(A)|+|N(B)|-|N(A)\cup N(B)|.$$
Since $N(A)\subseteq N(A\cap B)$ and $N(B)\subseteq N(A\cap B)$, $N(A)\cup N(B)\subseteq N(A\cap B)$, and so $-|N(A)\cup N(B)|\geq -|N(A\cap B)|$. Then $|N(A\cup B)|\geq |N(A)|+|N(B)|-|N(A\cap B)|$.
%Thus $|N(A\cap B)|\geq |N(A)\cup N(B)|=|N(A)|+|N(B)|-|N(A)\cap N(B)|=|N(A)|+|N(B)|-|N(A\cup B)|$.
Thus, $\widehat{N}(A\cup B)\geq\widehat{N}(A)+\widehat{N}(B)-\widehat{N}(A\cap B)$.
\end{proof}

The following two lemmas ensure that each variable of all programs in Sections \ref{sec:main_program} and \ref{sec:solve_program} is strictly positive.

\begin{Lemma}\label{lem:N(S)_2}
Let $G$ be an $n$-vertices graph with minimum degree $\delta(G)>\frac{4n}{5}$. If $S\subseteq V(G)$ such that $1\leq |S|\leq 5$, then
$$\widehat{N}(S)>1-\frac{|S|}{5}.$$
\end{Lemma}
\begin{proof}
Let $S=\{v_1,\dots,v_{|S|}\}$ and $d=1-\frac{\delta(G)}{n}<\frac{1}{5}$.
By the definition of $d$, $\widehat{N}(v)\geq1-d$ for any $v\in S$.
It follows from Lemma \ref{lem:N(S)_1} that $\widehat{N}(v_1,v_2)\geq\widehat{N}(v_1)+\widehat{N}(v_2)-\widehat{N}(\emptyset)\geq(1-d)+(1-d)-1=1-2d$.
By repeated applications of Lemma \ref{lem:N(S)_1}, $\widehat{N}(v_1,\dots,v_{|S|})\geq\widehat{N}(v_1,\dots,v_{|S|-1})+\widehat{N}(v_{|S|})-\widehat{N}(\emptyset)\geq[1-(|S|-1)d]+(1-d)-1=1-|S|d$.
Thus $\widehat{N}(S)\geq1-|S|d>1-\frac{|S|}{5}$.
\end{proof}
\begin{Lemma}\label{lem:N(S)_3}
Let $G$ be an $n$-vertices graph with minimum degree $\delta(G)>\frac{4n}{5}$. If $A,B\subseteq V(G)$ with $A\neq\emptyset$, $B\neq\emptyset$ and $|A\cup B|\leq 5$, then
$$\widehat{N}(A)+\widehat{N}(B)-\widehat{N}(A\cap B)>0.$$
\end{Lemma}
\begin{proof}
By Lemma \ref{lem:N(S)_2}, $\widehat{N}(A)>1-\frac{|A|}{5}\geq0$ since $|A|\leq |A\cup B|\leq 5$.

If $B\subseteq A$, then $\widehat{N}(B)=\widehat{N}(A\cap B)$, and so $\widehat{N}(A)+\widehat{N}(B)-\widehat{N}(A\cap B)=\widehat{N}(A)>0.$

Suppose that $B\setminus A\neq\emptyset$. It follows from Lemma \ref{lem:N(S)_1} that $\widehat{N}(B)\geq\widehat{N}(B\cap A)+\widehat{N}(B\setminus A)-\widehat{N}(\emptyset)$. That is, $$\widehat{N}(B)-\widehat{N}(A\cap B)\geq\widehat{N}(B\setminus A)-\widehat{N}(\emptyset)=\widehat{N}(B\setminus A)-1.$$
Since $\widehat{N}(B\setminus A)>1-\frac{|B\setminus A|}{5}$ by Lemma \ref{lem:N(S)_2}, $\widehat{N}(B)-\widehat{N}(A\cap B)>-\frac{|B\setminus A|}{5},$ and so
$$\widehat{N}(A)+\widehat{N}(B)-\widehat{N}(A\cap B)>\left(1-\frac{|A|}{5}\right)-\frac{|B\setminus A|}{5}=1-\frac{|A\cup B|}{5}.$$
By $|A\cup B|\leq 5$, $\widehat{N}(A)+\widehat{N}(B)-\widehat{N}(A\cap B)>1-\frac{|A\cup B|}{5}\geq0.$
\end{proof}

\subsection{Main programming}\label{sec:main_program}
Suppose that $G$ is a graph on $n$ vertices with minimum degree $\delta(G)>\frac{4n}{5}$, $d=1-\frac{\delta(G)}{n}$, $O=(x_1,x_2,x_3,x_4)\in\mathcal {OK}_4(G)$ and $R=N(x_1,x_2,x_3,x_4)$ in this section. To prove Theorem \ref{thm:W'_G(O)<=1}, we first formulate a programming with the objective function $W'_G(O)$. %Note that $d<\frac{1}{5}$.

For every $r\in\{2,3,4,5\}$ and every $K=(v_1,\dots,v_r)\in\mathcal {OK}_r(G)$, define a scaled weight
$$\widehat{W}(K)=n^{r-1}\cdot W(K)=\prod_{i=2}^{r}\frac{1}{\widehat{N}(v_1,\dots,v_r)}.$$
For ease of reading, we will let $\widehat{W}(v_1 ,\dots, v_r)=\widehat{W}(K)$. By the definition of $\widehat{W}(v_1 ,\dots, v_r)$, we will rewrite $W'_G(O)$ as follows:
\begin{align}\label{eq:W'_G(O)}
W'_G(O)=&\frac{1}{\widehat{W}(x_1,x_2,x_3)\cdot n}\sum_{y\in R}[2\widehat{W}(x_1,y,x_2,x_3)-\widehat{W}(x_1,x_2,x_3,y)-\widehat{W}(x_1,x_2,y,x_3)+\notag \\
&\sum_{z\in N(y)\cap R}\frac{1}{n}(2\widehat{W}(x_1,y,x_2,x_3,z)+2\widehat{W}(x_1,y,x_2,z,x_3)+2\widehat{W}(x_1,y,z,x_2,x_3)-\\
&\widehat{W}(x_1,x_2,x_3,y,z)-\widehat{W}(x_1,x_2,y,x_3,z)-\widehat{W}(x_1,x_2,y,z,x_3)-3\widehat{W}(y,z,x_1,x_2,x_3))]\notag
\end{align}
Without loss of generality, suppose that $R=\{y_1,y_2,\dots,y_{R_0}\}$ and $N(y_i)\cap R=\{z_{i,1},\dots,z_{i,R_i}\}$ for any $1\leq i\leq R_0$. Let us replace the neighborhood densities in (\ref{eq:W'_G(O)}) with variable names as follows:
\begin{enumerate}
\item[$\bullet$] $x=\widehat{N}(x_1)$,
\item[$\bullet$] $y'_i=\widehat{N}(y_i)$ for all $1\leq i\leq R_0$,
\item[$\bullet$] $e_0=\widehat{N}(x_1,x_2)$,
\item[$\bullet$] $e_i=\widehat{N}(x_1,y_i)$ for all $1\leq i\leq R_0$,
\item[$\bullet$] $f_{i,j}=\widehat{N}(y_i,z_{i,j})$ for all $1\leq i\leq R_0$ and $1\leq j\leq R_i$,
\item[$\bullet$] $g_0=\widehat{N}(x_1,x_2,x_3)$,
\item[$\bullet$] $q_{i,0}=\widehat{N}(x_1,x_2,y_i)$ for all $1\leq i\leq R_0$,
\item[$\bullet$] $q_{i,j}=\widehat{N}(x_1,y_i,z_{i,j})$ for all $1\leq i\leq R_0$ and $1\leq j\leq R_i$,
\item[$\bullet$] $p_{i,0}=\widehat{N}(x_1,x_2,x_3,y_i)$ for all $1\leq i\leq R_0$,
\item[$\bullet$] $p_{i,j}=\widehat{N}(x_1,x_2,y_i,z_{i,j})$ for all $1\leq i\leq R_0$ and $1\leq j\leq R_i$,
\item[$\bullet$] $h_{i,j}=\widehat{N}(x_1,x_2,x_3,y_i,z_{i,j})$ for all $1\leq i\leq R_0$ and $1\leq j\leq R_i$.
\end{enumerate}
So
%\begin{equation}\label{eq:W'_G(O)_1}
\begin{align*}
W'_G(O)=&\frac{e_0g_0}{n}\sum_{i=1}^{R_0}[\frac{2}{e_iq_{i,0}p_{i,0}}-\frac{1}{e_0g_0p_{i,0}}-\frac{1}{e_0q_{i,0}p_{i,0}}+
\sum_{j=1}^{R_i}\frac{1}{n}(\frac{2}{e_iq_{i,0}p_{i,0}h_{i,j}}+\frac{2}{e_iq_{i,0}p_{i,j}h_{i,j}}\\
&+\frac{2}{e_iq_{i,j}p_{i,j}h_{i,j}}-\frac{1}{e_0g_0p_{i,0}h_{i,j}}-\frac{1}{e_0q_{i,0}p_{i,0}h_{i,j}}-\frac{1}{e_0q_{i,0}p_{i,j}h_{i,j}}
-\frac{3}{f_{i,j}q_{i,j}p_{i,j}h_{i,j}})]\\
=&\frac{e_0g_0}{n}\sum_{i=1}^{R_0}[\frac{1}{p_{i,0}}(\frac{2}{e_iq_{i,0}}-\frac{1}{e_0g_0}-\frac{1}{e_0q_{i,0}})+
\sum_{j=1}^{R_i}\frac{1}{n\cdot h_{i,j}}(\frac{2}{e_iq_{i,0}p_{i,0}}+\frac{2}{e_iq_{i,0}p_{i,j}}+\\
&\frac{2}{e_iq_{i,j}p_{i,j}}-\frac{1}{e_0g_0p_{i,0}}-\frac{1}{e_0q_{i,0}p_{i,0}}-\frac{1}{e_0q_{i,0}p_{i,j}}
-\frac{3}{f_{i,j}q_{i,j}p_{i,j}})].
\end{align*}
%\end{equation}
By the definition of $x,y'_i$ and $d$, $x,y'_i\in[1-d,1]$ for any $1\leq i\leq R_0$. It follows from Lemmas \ref{lem:decrease_N(K)} and \ref{lem:N(S)_1} that
$$x-d\leq e_0\leq x,\text{    }x+y'_i-1\leq e_i\leq x, \text{    }y'_i-d\leq f_{i,j}\leq y'_i,$$
$$e_0-d\leq g_0\leq e_0,\text{    }e_0+e_i-x\leq q_{i,0}\leq e_0,\text{    }e_i+f_{i,j}-y'_i\leq q_{i,j}\leq e_i,$$
$$g_0+e_i-x\leq p_{i,0}\leq g_0,\text{    }q_{i,0}+f_{i,j}-y'_i\leq p_{i,j}\leq q_{i,0}\text{ and }p_{i,0}+q_{i,j}-e_{i}\leq h_{i,j}\leq p_{i,0}.$$
Note that $R_0=|N(v_1,v_2,v_3,v_4)|=n\cdot\widehat{N}(v_1,v_2,v_3,v_4)$ and $R_i=|N(v_1,v_2,v_3,v_4,y_i)|=n\cdot\widehat{N}(v_1,v_2,v_3,v_4,y_i)$. By Lemma \ref{lem:N(S)_2} and $\delta(G)>\frac{4n}{5}$, $\widehat{N}(v_1,v_2,v_3,v_4)>0$ and $\widehat{N}(v_1,v_2,v_3,v_4,y_i)>0$. It follows from Lemma \ref{lem:decrease_N(K)} that $\widehat{N}(v_1,v_2,v_3,v_4)\leq \widehat{N}(v_1,v_2,v_3)$ and $\widehat{N}(v_1,v_2,v_3,v_4,y_i)\leq\widehat{N}(v_1,v_2,v_3,y_i)$.
Thus $1\leq R_0\leq n\cdot g_0$ and $1\leq R_i\leq n\cdot p_{i,0}$.

Let $$W^{(1)}_G=W'_G(O).$$ Here is a programming:
%\begin{align*}
%W^{(1)}_G
%=\frac{e_0g_0}{n}\sum_{i=1}^{R_0}[&\frac{1}{p_{i,0}}(\frac{2}{e_iq_{i,0}}-\frac{1}{e_0g_0}-\frac{1}{e_0q_{i,0}})+
%\sum_{j=1}^{R_i}\frac{1}{n\cdot h_{i,j}}(\frac{2}{e_iq_{i,0}p_{i,0}}+\frac{2}{e_iq_{i,0}p_{i,j}}+\\
%&\frac{2}{e_iq_{i,j}p_{i,j}}-\frac{1}{e_0g_0p_{i,0}}-\frac{1}{e_0q_{i,0}p_{i,0}}-\frac{1}{e_0q_{i,0}p_{i,j}}
%-\frac{3}{f_{i,j}q_{i,j}p_{i,j}})].
%\end{align*}
%To prove $W'_G(O)\leq 1$, it suffices to prove that the following programming has value at most $1$.

% 在导言区添加: \usepackage{amsmath} 和 \documentclass[fleqn]{article}
\begin{align*}
&\text{(P1): maximize } W^{(1)}_G & \\
& \quad\text{s.t. for all } 1\leq i \leq R_0 \text{ and } 1\leq j \leq R_i & \\
& \qquad\text{I. Degree constraints:} & x \in [1 - d, 1], \\
                                    & & y'_i \in [1 - d, 1], \\
& \qquad\text{II. Triangle constraints:} & e_0 \in [x - d, x], \\
                                       & & e_i \in [x + y'_i - 1, x], \\
                                       & & f_{i,j} \in [y'_i - d, y'_i], \\
& \qquad\text{III. } K_4 \text{ constraints:} &g_0\in[e_0-d,e_0],\\
                                            & & q_{i,0} \in [e_i + e_0 - x, e_0], \\
                                            & & q_{i,j} \in [e_i + f_{i,j} - y'_i, e_i], \\
& \qquad\text{IV. } K_5 \text{ constraints:} & p_{i,0}\in[g_0+e_i-x,g_0],\\
                                           & & p_{i,j} \in [q_{i,0} + f_{i,j} - y'_i, q_{i,0}], \\
& \qquad\text{V. } K_6 \text{ constraints:} & h_{i,j}\in[p_{i,0}+q_{i,j}-e_{i},p_{i,0}],\\
& \qquad\text{VI. Number of terms constraints:} & R_0 \in [0, n\cdot g_0], \\
                                              & & R_i \in [0, n\cdot p_{i,0}].
\end{align*}

The following lemma ensures that $ W^{(1)}_G$ is well-defined in the domain of (P1).
\begin{Lemma}\label{lem:var>0_1}
Let $G$ be a graph on $n$ vertices with minimum degree $\delta(G)>\frac{4n}{5}$ and $d=1-\frac{\delta(g)}{n}$.
Then
\begin{enumerate}
\item[$(1)$] $1 - d>0$, $x - d>0$, $e_0-d>0$,
\item[$(2)$] $y'_i - d>0$, $x + y'_i - 1>0$, $e_i + e_0 - x>0$, $g_0+e_i-x>0$,
\item[$(3)$] $e_i + f_{i,j} - y'_i>0$, $q_{i,0} + f_{i,j} - y'_i>0$ and $p_{i,0}+q_{i,j}-e_{i}>0$.
\end{enumerate}
\end{Lemma}
\begin{proof}
By $\delta(G)>\frac{4n}{5}$ and $d=1-\frac{\delta(G)}{n}$, $x>\frac{4}{5}$ and $d<\frac{1}{5}$. Then $1 - d>0$ and $x - d>0$. Since $e_0 \in [x - d, x]$, $e_0-d\geq x-2d>0$. By the definition of $y'$, $y'>\frac{4}{5}$, and so $y'_i - d>0$. It follows from Lemma \ref{lem:N(S)_3} that $x + y'_i - 1>0$, $e_i + e_0 - x>0$, $g_0+e_i-x>0$, $e_i + f_{i,j} - y'_i>0$, $q_{i,0} + f_{i,j} - y'_i>0$ and $p_{i,0}+q_{i,j}-e_{i}>0$.
%$x + y'_i - 1=\widehat{N}(x_1)+\widehat{N}(y_i)-\widehat{N}(\emptyset)>0$, $e_i + e_0 - x=\widehat{N}(x_1,y_i)+\widehat{N}(x_1,x_2)-\widehat{N}(x_1)>0$, $e_i + f_{i,j} - y'_i=\widehat{N}(x_1,y_i)+\widehat{N}(y_i,z_{i,j})-\widehat{N}(y_i)>0$, $g_0+e_i-x=\widehat{N}(x_1,x_2,x_3)+\widehat{N}(x_1,y_i)-\widehat{N}(x_1)>0$, $q_{i,0} + f_{i,j} - y'_i=\widehat{N}(x_1,x_2,y_i)+\widehat{N}(y_i,z_{i,j})-\widehat{N}(y_i)>0$ and $p_{i,0}+q_{i,j}-e_{i}=\widehat{N}(x_1,x_2,x_3,y_i)+\widehat{N}(x_1,x_2,y_i,z_{i,j})-\widehat{N}(x_1,y_i)>0$.
\end{proof}

To prove Theorem \ref{thm:W'_G(O)<=1}, i.e., $W'_G(O)\leq 1$, it suffices to prove that the programming {\rm (P1)} has value at most $1$. By the above analysis, we obtain the following lemma.
\begin{Lemma}\label{lem:program->thm2.5}
If the maximum value of {\rm (P1)} is at most $1$, then $W'_G(O)\leq 1$ for any $O\in\mathcal {OK}_4(G)$.
\end{Lemma}

To reduce the number of variables, we need the following Weierstrass' Theorem (see \cite{Chachuat}).
\begin{Theorem}[Weierstrass' Theorem]\label{thm:Weierstrass_thm}
Let $S$ be a nonempty, closed and bounded set. If $f:S\rightarrow\mathbb{R}$ is continuous on $S$, then the problem $\max\{f(x):x\in S\}$ attains its global maximum, that is, there exists a maximizing solution to this problem.
\end{Theorem}

\begin{Lemma}\label{lem:(P1)}
The maximum value of {\rm (P1)} is achieved when for all $1\leq i\leq R_0$ and $1\leq j,j'\leq R_i$, we have
$$f_{i,j}=f_{i,j'},\text{    }q_{i,j}=q_{i,j'},\text{    }p_{i,j}=p_{i,j'}\text{  and    }h_{i,j}=h_{i,j'}.$$
\end{Lemma}
\begin{proof}
%By Lemma \ref{lem:var>0_1}, $W_G^{(1)}$ is well-defined on the domain of (P1).
Note that $W_G^{(1)}$ is continuous on the domain of (P1).
Since the domain of (P1) is closed and bounded, we find that (P1) has a global maximum by Theorem \ref{thm:Weierstrass_thm}.

Let $P_0$ be a point that achieves this maximum. For the above $P_0$ and each $i$, let $1\leq j_i\leq R_i$ such that
$$\frac{1}{h_{i,j_i}}(\frac{2}{e_iq_{i,0}p_{i,0}}+\frac{2}{e_iq_{i,0}p_{i,j_i}}+
\frac{2}{e_iq_{i,j_i}p_{i,j_i}}-\frac{1}{e_0g_0p_{i,0}}-\frac{1}{e_0q_{i,0}p_{i,0}}-\frac{1}{e_0q_{i,0}p_{i,j_i}}
-\frac{3}{f_{i,j_i}q_{i,j_i}p_{i,j_i}})$$
is maximized over all $1\leq j\leq R_i$. Then the $P'_0$ obtained from $P_0$ by setting $f_{i,j}=f_{i,j_i}$, $q_{i,j}=q_{i,j_i}$, $p_{i,j}=p_{i,j_i}$ and $h_{i,j}=h_{i,j_i}$ for all $1\leq i\leq R_0$ and $1\leq j\leq R_i$ is also a point that achieves this maximum. Moreover, since the constraints for the $f_{i,j}$, $q_{i,j}$, $p_{i,j}$ and $h_{i,j}$ are identical for each $1\leq j\leq R_i$, $P'_0$ also satisfies the constraints of (P1) as desired.
\end{proof}

By Lemma \ref{lem:(P1)}, let $$f_i=f_{i,1}, q_i=q_{i,1}, p_i=p_{i,1} \ \text{and}\  h_i=h_{i,1}$$ without loss of generality. Let $r_i=\frac{R_i}{n}$ and we form a new programming {\rm (P2)} with a new objective function that has the same optimum value as {\rm (P1)}:
\begin{align*}
W^{(2)}_G
=\frac{e_0g_0}{n}\sum_{i=1}^{R_0}[&\frac{1}{p_{i,0}}(\frac{2}{e_iq_{i,0}}-\frac{1}{e_0g_0}-\frac{1}{e_0q_{i,0}})+
\frac{r_i}{h_{i}}(\frac{2}{e_iq_{i,0}p_{i,0}}+\frac{2}{e_iq_{i,0}p_{i}}+\\
&\frac{2}{e_iq_{i}p_{i}}-\frac{1}{e_0g_0p_{i,0}}-\frac{1}{e_0q_{i,0}p_{i,0}}-\frac{1}{e_0q_{i,0}p_{i}}
-\frac{3}{f_{i}q_{i}p_{i}})].
\end{align*}
The new programming is as follows:
\begin{align*}
&\text{(P2): maximize } W^{(2)}_G \\
& \quad\text{s.t. for all } 1\leq i \leq R_0 \\
& \qquad\text{I. Degree constraints:} & x \in [1 - d, 1], \\
                                    & & y'_i \in [1 - d, 1], \\
& \qquad\text{II. Triangle constraints:} & e_0 \in [x - d, x], \\
                                       & & e_i \in [x + y'_i - 1, x], \\
                                       & & f_{i} \in [y'_i - d, y'_i], \\
& \qquad\text{III. } K_4 \text{ constraints:} &g_0\in[e_0-d,e_0],\\
                                            & & q_{i,0} \in [e_i + e_0 - x, e_0], \\
                                            & & q_{i} \in [e_i + f_{i} - y'_i, e_i], \\
& \qquad\text{IV. } K_5 \text{ constraints:} & p_{i,0}\in[g_0+e_i-x,g_0],\\
                                           & & p_{i} \in [q_{i,0} + f_{i} - y'_i, q_{i,0}], \\
& \qquad\text{V. } K_6 \text{ constraints:} & h_{i}\in[p_{i,0}+q_{i}-e_{i},p_{i,0}],\\
& \qquad\text{VI. Number of terms constraints:} & R_0 \in [0, n\cdot g_0], \\
                                              & & r_i \in [0, p_{i,0}].
\end{align*}
By Lemma \ref{lem:var>0_1} and the definition of $f_i$, $q_i$, $p_i$ and $h_i$, $W^{(2)}_G$ is well-defined in the domain of (P2).
%It is clear that $W^{(2)}_G$ is well-defined in the domain of (P2).
\begin{Corollary}\label{cor:P1=P2}
{\rm (P1)} and {\rm (P2)} have the same global maximum. That is, {\rm OPT(P1)=OPT(P2)}.
\end{Corollary}

\begin{Lemma}\label{lem:(P2)}
The maximum value of {\rm (P2)} is achieved when for all $1\leq i,i'\leq R_0$, we have
$$y'_i=y'_{i'},\text{    }e_i=e_{i'},\text{    }f_i=f_{i'},\text{    }q_{i,0}=q_{i',0},\text{    }q_i=q_{i'},\text{    }p_{i,0}=p_{i',0},\text{    }p_{i}=p_{i'},\text{    }h_{i}=h_{i'},\text{    }r_i=r_{i'}.$$
\end{Lemma}
\begin{proof}
By the similar proof of Lemma \ref{lem:(P1)}, (P2) has a global maximum by Theorem \ref{thm:Weierstrass_thm}.
Let $P_0$ be a point in the domain of (P2) that achieves the maximum of (P2). For the above point $P_0$, let $1\leq I\leq R_0$ be an integer such that
$$\frac{1}{p_{i,0}}(\frac{2}{e_iq_{i,0}}-\frac{1}{e_0g_0}-\frac{1}{e_0q_{i,0}})+
\frac{r_i}{h_{i}}(\frac{2}{e_iq_{i,0}p_{i,0}}+\frac{2}{e_iq_{i,0}p_{i}}+
\frac{2}{e_iq_{i}p_{i}}-\frac{1}{e_0g_0p_{i,0}}-\frac{1}{e_0q_{i,0}p_{i,0}}-\frac{1}{e_0q_{i,0}p_{i}}
-\frac{3}{f_{i}q_{i}p_{i}})$$
is maximized over all $1\leq i\leq R_0$. Then the point $P'_0$ obtained from $P_0$ by setting $y'_i=y'_{I},$
$e_i=e_{I},f_i=f_{I},q_{i,0}=q_{I,0},q_i=q_{I},p_{i,0}=p_{I,0},p_{i}=p_{I},h_{i}=h_{I}$ and $r_i=r_{I}$ for all $1\leq i\leq R_0$ is also a point that achieves this maximum. Moreover, since the constraints for the $y'_i,e_i,f_i,q_{i,0},q_i,p_{i,0},p_{i},h_{i}$ and $r_i$ are identical for any $1\leq i\leq R_0$, $P'_0$ satisfies the constraints of (P2) as desired.
\end{proof}

By Lemma \ref{lem:(P2)}, let $$y'=y'_{1},\ e=e_{1},\ f=f_{1},\ q_{0}=q_{1,0},\ q=q_{1},\ p_{0}=p_{1,0},\ p=p_{1},\ h=h_{1}\ \text{and}\ r=r_{1},$$ without loss of generality. Let $r_0=\frac{R_0}{n}$ and we form a new programming (P3) with a new objective function that has the same optimum value as (P2) and hence as (P1):
\begin{align*}
W^{(3)}_G=e_0g_0r_0[&\frac{1}{p_{0}}(\frac{2}{eq_{0}}-\frac{1}{e_0g_0}-\frac{1}{e_0q_{0}})+
\frac{r}{h}(\frac{2}{eq_{0}p_{0}}+\frac{2}{eq_{0}p}+\frac{2}{eqp}-\frac{1}{e_0g_0p_{0}}-\frac{1}{e_0q_{0}p_{0}}\\
&-\frac{1}{e_0q_{0}p}-\frac{3}{fqp})].
\end{align*}
Here is the new programming:
\begin{align*}
&\text{(P3): maximize } W^{(3)}_G \\
& \quad\text{s.t.}\\
%& \quad\text{s.t. for all } 1\leq i \leq R_0 \\
& \qquad\text{I. Degree constraints:} & x \in [1 - d, 1], \\
                                    & & y' \in [1 - d, 1], \\
& \qquad\text{II. Triangle constraints:} & e_0 \in [x - d, x], \\
                                       & & e \in [x + y' - 1, x], \\
                                       & & f \in [y' - d, y'], \\
& \qquad\text{III. } K_4 \text{ constraints:} &g_0\in[e_0-d,e_0],\\
                                            & & q_{0} \in [e + e_0 - x, e_0], \\
                                            & & q \in [e + f - y', e], \\
& \qquad\text{IV. } K_5 \text{ constraints:} & p_{0}\in[g_0+e-x,g_0],\\
                                           & & p \in [q_{0} + f - y', q_0], \\
& \qquad\text{V. } K_6 \text{ constraints:} & h \in[p_{0}+q-e,p_{0}],\\
& \qquad\text{VI. Number of terms constraints:} & r_0 \in [0, g_0], \\
                                              & & r \in [0, p_{0}].
\end{align*}
Similarly, Lemma \ref{lem:var>0_1} ensures that $ W^{(3)}_G$ is well-defined in the domain of (P3).
\begin{Corollary}\label{cor:P1=P3}
{\rm (P1)} and {\rm (P3)} have the same global maximum. That is, {\rm OPT(P1)=OPT(P3)}.
\end{Corollary}
\begin{proof}
By Lemma \ref{lem:(P2)} and Corollary \ref{cor:P1=P2}, we obtain the desired conclusion.
\end{proof}

%%%%%%%%%%%%%%%%%%%%%%%%%%%%%%%%%%%%%%%%%%%%%分章节%%%%%%%%%%%%%%%%%%%%%%%%%%%%%%%%%%%%%%%%

\section{Solving the programming}\label{sec:solve_program}

We do not actually solve (P3), rather in Subsection \ref{sec:7_variables}, we upper bound (P3) with a new program (P4) that uses ramp functions.

For a real-valued function $w(u)$ where $u\in\mathbb{R}^n$, define a {\em ramp function} $w^+(u)$ of $w(u)$ as following:
\[ w^+(u)=\frac{w(u)+|w(u)|}{2}= \left\{
\begin{array}{cl}
w(u), & \text{if } w(u)\geq0, \text{ and }\\
0, & \text{otherwise}.
\end{array} \right. \]
Note that $w^+(u)\geq w(u)$ and if $w$ is a real-valued function that is continuous on a region $\mathbb{R}$, then $w^+$ is continuous on $\mathbb{R}$.

We will slowly reduce the number of variables by constructing a new programming with a larger optimum value, first to nine (P5), then to seven (P6), six (P7), five (P8), four (P9), three (P10), two (P11), one (P12) and then we actually find the upper bound $W_G^{(13)}$ of the optimal value of (P12). Finally we prove $W_G^{(13)}\leq 1$ for $d\in[0,\frac{2}{33}]$ as required.

Suppose that $0\leq d\leq \frac{2}{33}$ in this section. So all programs in this section is well-defined. By Theorem \ref{thm:Weierstrass_thm}, each of them has a maximizing solution.

\subsection{Reduction to seven variables}\label{sec:7_variables}

We will modify the objective function $W^{(3)}_G$ by ramp functions. Recall that
\begin{align*}
W^{(3)}_G=e_0g_0r_0[&\frac{1}{p_{0}}(\frac{2}{eq_{0}}-\frac{1}{e_0g_0}-\frac{1}{e_0q_{0}})+
\frac{r}{h}(\frac{2}{eq_{0}p_{0}}+\frac{2}{eq_{0}p}+\frac{2}{eqp}-\frac{1}{e_0g_0p_{0}}-\frac{1}{e_0q_{0}p_{0}}\\
&-\frac{1}{e_0q_{0}p}-\frac{3}{fqp})]\\
=e_0g_0r_0[&(\frac{1}{p_0}+\frac{r}{hp_0})(\frac{2}{eq_0}-\frac{1}{e_0g_0}-\frac{1}{e_0q_0})+\frac{r}{h}(\frac{2}{eq_{0}p}+\frac{2}{eqp}
-\frac{1}{e_0q_{0}p}-\frac{3}{fqp})]\\
=e_0g_0r_0\{&(\frac{1}{p_0}+\frac{r}{hp_0})(\frac{2}{eq_0}-\frac{1}{e_0g_0}-\frac{1}{e_0q_0})+\frac{r}{hp}[\frac{1}{q_0}(\frac{1}{e}-\frac{1}{e_0})+
\frac{1}{q}(\frac{3}{e}-\frac{3}{f})+\text{  }\frac{1}{e}(\frac{1}{q_0}-\frac{1}{q})]\}\\
=e_0g_0r_0\{&(\frac{1}{p_0}+\frac{r}{hp_0})[(\frac{1}{q_0}+\frac{1}{g_0})(\frac{1}{e}-\frac{1}{e_0})+\frac{1}{e}(\frac{1}{q_0}-\frac{1}{g_0})]+\frac{r}{hp}[\frac{1}{q_0}(\frac{1}{e}-\frac{1}{e_0})+
\frac{1}{q}(\frac{3}{e}-\frac{3}{f})+\\
&\text{  }\frac{1}{e}(\frac{1}{q_0}-\frac{1}{q})]\}.
\end{align*}
Let
\begin{align*}
W^{(4)}_G
=e_0g_0r_0\{&(\frac{1}{p_0}+\frac{r}{hp_0})[(\frac{1}{q_0}+\frac{1}{g_0})\frac{(e_0-e)^+}{ee_0}+\frac{(g_0-q_0)^+}{eq_0g_0}]+\frac{r}{hp}[
\frac{(e_0-e)^+}{q_0ee_0}+\frac{3(f-e)^+}{qef}\\
&+\frac{(q-q_0)^+}{eq_0q}]\}.
\end{align*}
Then we obtain a new programming and a lemma as following:
\begin{align*}
&\text{(P4): maximize } W^{(4)}_G \\
& \quad\text{s.t.}\\
%& \quad\text{s.t. for all } 1\leq i \leq R_0 \\
& \qquad\text{I. Degree constraints:} & x \in [1 - d, 1], \\
                                    & & y' \in [1 - d, 1], \\
& \qquad\text{II. Triangle constraints:} & e_0 \in [x - d, x], \\
                                       & & e \in [x + y' - 1, x], \\
                                       & & f \in [y' - d, y'], \\
& \qquad\text{III. } K_4 \text{ constraints:} &g_0\in[e_0-d,e_0],\\
                                            & & q_{0} \in [e + e_0 - x, e_0], \\
                                            & & q \in [e + f - y', e], \\
& \qquad\text{IV. } K_5 \text{ constraints:} & p_{0}\in[g_0+e-x,g_0],\\
                                           & & p \in [q_{0} + f - y', q_0], \\
& \qquad\text{V. } K_6 \text{ constraints:} & h \in[p_{0}+q-e,p_{0}],\\
& \qquad\text{VI. Number of terms constraints:} & r_0 \in [0, g_0], \\
                                              & & r \in [0, p_{0}].
\end{align*}
\begin{Lemma}\label{lem:P4>=P3}
 {\rm OPT(P1)=OPT(P3)$\leq$ OPT(P4)}.
\end{Lemma}
\begin{proof}
%Since $w^+(u)\geq w(u)$, $W^{(4)}_G\geq W^{(3)}_G$.
By the definition of a ramp function, $w^+(u)\geq w(u)$ for any function $w(u)$, and so $W^{(4)}_G\geq W^{(3)}_G$.
Since the domain of (P4) is the same as the domain of (P3), OPT(P3)$\leq$ OPT(P4). By Corollary \ref{cor:P1=P3}, the desired conclusion is obtained.
\end{proof}

\begin{Lemma}\label{lem:P4_maximum}
The maximum of (P4) is achieved when $r=p_0$, $r_0=g_0$, $h=p_0+q-e$ and $p=q_0+f-y'$ hold.
\end{Lemma}
\begin{proof}
Let $P_0=(x,y',e_0,e,f,g_0,q_0,q,p_0,p,h,r_0,r)$ be a point that achieves the maximum of (P4). Let $P'_0=(x,y',e_0,e,f,g_0,q_0,q,p_0,p',h',r'_0,r')$, where $r'=p_0$, $r'_0=g_0$, $h'=p_0+q-e$ and $p'=q_0+f-y'$.
Note that $P'_0$ is in the domain of (P4).

It follows from Lemma \ref{lem:var>0_1} and the definitions of all variables that $e_0,e,f,g_0,q_0,q,p_0>0$, $h\geq p_0+q-e>0$ and $p\geq q_0+f-y'>0$ in the domain of (P4). Since $P_0$ is in the domain of (P4), $r'\geq r$, $r'_0\geq r_0$, $h'\leq h$ and $p'\leq p$. Then $0<\frac{r}{h}\leq\frac{r'}{h'}$ and $0<\frac{r}{hp}\leq\frac{r'}{h'p'}$. Since all ramp functions are non-negative,
$$(\frac{1}{p_0}+\frac{r}{hp_0})[(\frac{1}{q_0}+\frac{1}{g_0})\frac{(e_0-e)^+}{ee_0}+\frac{(g_0-q_0)^+}{eq_0g_0}]\leq (\frac{1}{p_0}+\frac{r'}{h'p_0})[(\frac{1}{q_0}+\frac{1}{g_0})\frac{(e_0-e)^+}{ee_0}+\frac{(g_0-q_0)^+}{eq_0g_0}]$$
and
$$\frac{r}{hp}[\frac{(e_0-e)^+}{q_0ee_0}+\frac{3(f-e)^+}{qef}+\frac{(q-q_0)^+}{eq_0q}]\leq
\frac{r'}{h'p'}[\frac{(e_0-e)^+}{q_0ee_0}+\frac{3(f-e)^+}{qef}+\frac{(q-q_0)^+}{eq_0q}].$$
Recall that $r'_0\geq r_0$. Thus $W^{(4)}_G(P'_0)\geq W^{(4)}_G(P_0)$.

So $P'_0$ is a point with $r=p_0$, $r_0=g_0$, $h=p_0+q-e$ and $p=q_0+f-y'$ that achieves the maximum of (P4).
\end{proof}

Let
\begin{align*}
W^{(5)}_G
=e_0g_0^2\{&(\frac{1}{p_0}+\frac{1}{p_0+q-e})[(\frac{1}{q_0}+\frac{1}{g_0})\frac{(e_0-e)^+}{ee_0}+\frac{(g_0-q_0)^+}{eq_0g_0}]+\\
&\frac{p_0}{(p_0+q-e)(q_0+f-y')}[\frac{(e_0-e)^+}{q_0ee_0}+\frac{3(f-e)^+}{qef}+\frac{(q-q_0)^+}{eq_0q}]\}.
\end{align*}
Thus we construct a new programming (P5) with the new objective function $W^{(5)}_G$ and a subset of the previous constraints as follows:
\begin{align*}
&\text{(P5): maximize } W^{(5)}_G \\
& \quad\text{s.t.}\\
%& \quad\text{s.t. for all } 1\leq i \leq R_0 \\
& \qquad\text{I. Degree constraints:} & x \in [1 - d, 1], \\
                                    & & y' \in [1 - d, 1], \\
& \qquad\text{II. Triangle constraints:} & e_0 \in [x - d, x], \\
                                       & & e \in [x + y' - 1, x], \\
                                       & & f \in [y' - d, y'], \\
& \qquad\text{III. } K_4 \text{ constraints:} &g_0\in[e_0-d,e_0],\\
                                            & & q_{0} \in [e + e_0 - x, e_0], \\
                                            & & q \in [e + f - y', e], \\
& \qquad\text{IV. } K_5 \text{ constraints:} & p_{0}\in[g_0+e-x,g_0].
\end{align*}
\begin{Corollary}\label{cor:P5>=P4}
 {\rm OPT(P1)$\leq$ OPT(P4)= OPT(P5)}.
\end{Corollary}
\begin{proof}
By Lemma \ref{lem:P4_maximum}, {\rm OPT(P4)= OPT(P5)}. Lemma \ref{lem:P4>=P3} implies that {\rm OPT(P1)$\leq$ OPT(P4)}. Then the desired conclusion holds.
%It follows from Lemmas \ref{lem:P4>=P3} and \ref{lem:P4_maximum} that the desired conclusion holds.
\end{proof}

\begin{Lemma}\label{lem:P5_maximum}
The maximum of (P5) is achieved when $p_0=g_0+e-x$ and $q_0=e_0+e-x$ hold.
\end{Lemma}
\begin{proof}
Let $P_0=(x,y',e_0,e,f,g_0,q_0,q,p_0)$ be a point that achieves the maximum of (P5). Let $P'_0=(x,y',e_0,e,f,g_0,q'_0,q,p'_0)$, where $p'_0=g_0+e-x$ and $q'_0=e_0+e-x$.
It is clear that $P'_0$ is in the domain of (P5).

By Lemma \ref{lem:var>0_1} and the definitions of all variables, $e_0,e,f,g_0,q'_0,q,p'_0>0$ in the domain of (P5). Since $P_0$ is in the domain of (P4), $p_0\geq p'_0$ and $q_0\geq q'_0$. For $d<\frac{1}{5}$, we know that
$$p'_0+q-e=g_0+q-x\geq e_0-d+e+f-y'-x\geq x+y'-3d-1\geq 1-5d>0,$$
where the first inequality holds by $g_0\geq e_0-d$ and $q\geq e+f-y'$, and the second inequality holds by $e_0\geq x-d$, $e\geq x+y'-1$ and $f\geq y'-d$, and the penultimate inequality holds by $x,y'\geq 1-d$. Then $\frac{1}{p_0+q-e}\leq \frac{1}{p'_0+q-e}$.
Therefore,
$$0<\frac{1}{p_0}+\frac{1}{p_0+q-e}\leq \frac{1}{p'_0}+\frac{1}{p'_0+q-e}.$$
Since $(e_0-e)^+\geq0$, $0\leq(\frac{1}{q_0}+\frac{1}{g_0})\frac{(e_0-e)^+}{ee_0}\leq (\frac{1}{q'_0}+\frac{1}{g_0})\frac{(e_0-e)^+}{ee_0}.$

Claim that $$0\leq\frac{(g_0-q_0)^+}{eq_0g_0}\leq\frac{(g_0-q'_0)^+}{eq'_0g_0}.$$
Indeed, if $g_0\leq q_0$, then $\frac{(g_0-q_0)^+}{eq_0g_0}=0$, and so $\frac{(g_0-q_0)^+}{eq_0g_0}\leq\frac{(g_0-q'_0)^+}{eq'_0g_0}$. If $g_0>q_0$, then $0<\frac{(g_0-q_0)^+}{eq_0g_0}=\frac{1}{eq_0}-\frac{1}{eg_0}\leq \frac{1}{eq'_0}-\frac{1}{eg_0}\leq \frac{(g_0-q'_0)^+}{eq'_0g_0}$.

Since $f\geq y'-d$, $e\geq x+y'-1$ and $e_0\geq x-d$, $q'_0+f-y'=e+e_0-x+f-y'\geq x+y'-1-2d$.
By $x,y'\geq 1-d$ and $d<\frac{1}{5}$, $x+y'-1-2d\geq 1-4d>0$. That is, $q'_0+f-y'=e+e_0-x+f-y'>0$. Since $q_0\geq q'_0$,
$$0<\frac{1}{q_0+f-y'}\leq \frac{1}{q'_0+f-y'}.$$
Since $e\geq q$, $p_0(q-e)\leq p'_0(q-e)$ and so $p_0(p'_0+q-e)\leq p'_0(p_0+q-e)$. Recall that $p_0+q-e\geq p'_0+q-e>0$. Thus
%$$\frac{p_0}{p_0+q-e}=\frac{1}{1-\frac{e-q}{p_0}}\leq \frac{1}{1-\frac{e-q}{p'_0}}=\frac{p'_0}{p'_0+q-e}.$$
$$0<\frac{p_0}{p_0+q-e}\leq \frac{p'_0}{p'_0+q-e}.$$

Since all ramp functions are nonnegative, $0\leq\frac{(e_0-e)^+}{q_0ee_0}+\frac{3(f-e)^+}{qef}\leq\frac{(e_0-e)^+}{q_0ee_0}+\frac{3(f-e)^+}{qef}$. Claim that $$0\leq\frac{(q-q_0)^+}{eq_0q}\leq \frac{(q-q'_0)^+}{eq'_0q}.$$
Indeed, if $q\leq q_0$, then $\frac{(q-q_0)^+}{eq_0q}=0$, and so $\frac{(q-q_0)^+}{eq_0q}\leq \frac{(q-q'_0)^+}{eq'_0q}.$ If $q>q_0$, then $\frac{(q-q_0)^+}{eq_0q}=\frac{1}{eq_0}-\frac{1}{eq}\leq \frac{1}{eq'_0}-\frac{1}{eq}=\frac{(q-q'_0)^+}{eq'_0q}.$
So $$0\leq\frac{(e_0-e)^+}{q_0ee_0}+\frac{3(f-e)^+}{qef}+\frac{(q-q_0)^+}{eq_0q}\leq\frac{(e_0-e)^+}{q_0ee_0}+\frac{3(f-e)^+}{qef}+
\frac{(q-q'_0)^+}{eq'_0q}.$$

To sum up, $W^{(5)}_G(P'_0)\geq W^{(5)}_G(P_0)$. Thus $P'_0$ is a optimal solution that $p_0=g_0+e-x$ and $q_0=e_0+e-x$.
\end{proof}

Let
\begin{align*}
W^{(6)}_G
=&e_0g_0^2\{(\frac{1}{g_0+e-x}+\frac{1}{g_0+q-x})[(\frac{1}{e+e_0-x}+\frac{1}{g_0})\frac{(e_0-e)^+}{ee_0}+\frac{(g_0-e-e_0+x)^+}{e(e+e_0-x)g_0}]+\\
&\frac{g_0+e-x}{(g_0+q-x)(e+e_0-x+f-y')}[\frac{(e_0-e)^+}{(e+e_0-x)ee_0}+\frac{3(f-e)^+}{qef}+\frac{(q-e-e_0+x)^+}{e(e+e_0-x)q}]\}.
\end{align*}
Thus we may replace (P5) with a new programming (P6) whose maximum value is at least that of (P5) as follows:
%Thus we construct a new programming (P6) with the following new objective function and a subset of the previous constraints as follows:
\begin{align*}
&\text{(P6): maximize } W^{(6)}_G \\
& \quad\text{s.t.}\\
%& \quad\text{s.t. for all } 1\leq i \leq R_0 \\
& \qquad\text{I. Degree constraints:} & x \in [1 - d, 1], \\
                                    & & y' \in [1 - d, 1], \\
& \qquad\text{II. Triangle constraints:} & e_0 \in [x - d, x], \\
                                       & & e \in [x + y' - 1, x], \\
                                       & & f \in [y' - d, y'], \\
& \qquad\text{III. } K_4 \text{ constraints:} &g_0\in[e_0-d,e_0],\\
                                            & & q \in [e + f - y', e].
\end{align*}
\begin{Corollary}\label{cor:P6>=P5}
 {\rm OPT(P1)$\leq$ OPT(P5)= OPT(P6)}.
\end{Corollary}
\begin{proof}
It follows from Corollary \ref{cor:P5>=P4} and Lemma \ref{lem:P5_maximum} that the desired conclusion holds.
\end{proof}

%%%%%%%%%%%%%%%%%%%%%%%%%%%分章节%%%%%%%%%%%%%%%%%%%%%%%%%%%%%%%%%%%%%%%%%%%%
\subsection{Reduction to four variables}\label{sec:4_variables}
Let
\begin{align*}
W^{(7)}_G
=e_0g_0^2\{&(\frac{1}{g_0+e-x}+\frac{1}{g_0+e+f-y'-x})[(\frac{1}{e+e_0-x}+\frac{1}{g_0})\frac{(e_0-e)^+}{ee_0}+\\
&\frac{(g_0-e-e_0+x)^+}{e(e+e_0-x)g_0}]+\frac{g_0+e-x}{(g_0+e+f-y'-x)(e+e_0-x+f-y')}\\
&[\frac{(e_0-e)^+}{(e+e_0-x)ee_0}+\frac{3(f-e)^+}{(e+f-y')ef}+\frac{x-e_0}{e^2(e+e_0-x)}]\}.
\end{align*}
Here is the new programming:
\begin{align*}
&\text{(P7): maximize } W^{(7)}_G \\
& \quad\text{s.t.}\\
%& \quad\text{s.t. for all } 1\leq i \leq R_0 \\
& \qquad\text{I. Degree constraints:} & x \in [1 - d, 1], \\
                                    & & y' \in [1 - d, 1], \\
& \qquad\text{II. Triangle constraints:} & e_0 \in [x - d, x], \\
                                       & & e \in [x + y' - 1, x], \\
                                       & & f \in [y' - d, y'], \\
& \qquad\text{III. } K_4 \text{ constraints:} &g_0\in[e_0-d,e_0].
\end{align*}
\begin{Lemma}\label{lem:P7>=P6}
 {\rm OPT(P7)$\geq$ OPT(P6)$\geq$ OPT(P1)}.
\end{Lemma}
\begin{proof}
By Lemma \ref{lem:var>0_1}, $e_0,e,f,g_0>0$, $g_0+e-x>0$, $e+e_0-x>0$ and $e+f-y'>0$.

By the domain of (P6), $g_0+q-x\geq g_0+e+f-y'-x\geq e_0+y'-1-2d\geq x+y'-1-3d\geq 1-5d>0$ because of $d<\frac{1}{5}$.
Since $q \in [e + f - y', e]$, $0<\frac{1}{g_0+q-x}\leq \frac{1}{g_0+e+f-y'-x}$, and so $$0<\frac{1}{g_0+e-x}+\frac{1}{g_0+q-x}\leq \frac{1}{g_0+e-x}+ \frac{1}{g_0+e+f-y'-x}.$$
By the definition of a ramp function, $(\frac{1}{e+e_0-x}+\frac{1}{g_0})\frac{(e_0-e)^+}{ee_0}+\frac{(g_0-e-e_0+x)^+}{e(e+e_0-x)g_0}\geq0$.
%By $g_0>0$, $e+e_0-x>0$ and the definition of a ramp function, $(\frac{1}{e+e_0-x}+\frac{1}{g_0})\frac{(e_0-e)^+}{ee_0}+\frac{(g_0-e-e_0+x)^+}{e(e+e_0-x)g_0}\geq0$.

Note that $e+e_0-x+f-y'\geq x+y'-2d-1\geq1-4d>0$ by $d<\frac{1}{4}$. Since $g_0+e-x>0$ and $q\geq e + f - y'$, $$\frac{g_0+e-x}{(g_0+e+f-y'-x)(e+e_0-x+f-y')}\geq\frac{g_0+e-x}{(g_0+q-x)(e+e_0-x+f-y')}>0.$$

Since all ramp functions are nonnegative and $q\geq e+f-y'$, $\frac{(e_0-e)^+}{(e+e_0-x)ee_0}+\frac{3(f-e)^+}{(e+f-y')ef}\geq \frac{(e_0-e)^+}{(e+e_0-x)ee_0}+\frac{3(f-e)^+}{qef}\geq0$.
%Since $e+e_0-x>0$ and $e+f-y'>0$ by Lemma \ref{lem:var>0_1}, $\frac{(e_0-e)^+}{(e+e_0-x)ee_0}+\frac{3(f-e)^+}{(e+f-y')ef}\geq \frac{(e_0-e)^+}{(e+e_0-x)ee_0}+\frac{3(f-e)^+}{qef}\geq0$ by the definition of a ramp function.

Claim that $$\frac{x-e_0}{e^2(e+e_0-x)}\geq \frac{(q-e-e_0+x)^+}{e(e+e_0-x)q}\geq0.$$
Indeed, if $q\leq e+e_0-x$, then $(q-e-e_0+x)^+=0$, and so $\frac{(q-e-e_0+x)^+}{e(e+e_0-x)q}=0\leq \frac{x-e_0}{e^2(e+e_0-x)}$ by $e_0\in[x-d,x]$.
Suppose that $q>e+e_0-x$. It follows from $q\leq e$ that $\frac{(q-e-e_0+x)^+}{e(e+e_0-x)q}=\frac{1}{e}(\frac{1}{e+e_0-x}-\frac{1}{q})\leq \frac{1}{e}(\frac{1}{e+e_0-x}-\frac{1}{e})=\frac{x-e_0}{e^2(e+e_0-x)}$.

Therefore, $W^{(7)}_G\geq W^{(6)}_G$. Since the constraints of (P7) are a subset of the constraints of (P6), OPT(P7)$\geq$ OPT(P6). The desired conclusion is obtained by Corollary \ref{cor:P6>=P5}.
\end{proof}

Recall that
\begin{align*}
W^{(7)}_G
=e_0\{&(\frac{g_0^2}{g_0+e-x}+\frac{g_0^2}{g_0+e+f-y'-x})[(\frac{1}{e+e_0-x}+\frac{1}{g_0})\frac{(e_0-e)^+}{ee_0}+\\
&\frac{(g_0-e-e_0+x)^+}{e(e+e_0-x)g_0}]+\frac{g_0^2(g_0+e-x)}{(g_0+e+f-y'-x)(e+e_0-x+f-y')}\\
&[\frac{(e_0-e)^+}{(e+e_0-x)ee_0}+\frac{3(f-e)^+}{(e+f-y')ef}+\frac{x-e_0}{e^2(e+e_0-x)}]\}.
\end{align*}
Let
\begin{align*}
W^{(8)}_G
=&e_0\{(\frac{e_0^2}{e_0+e-x}+\frac{e_0^2}{e_0+e+f-y'-x})[(\frac{1}{e+e_0-x}+\frac{1}{e_0-d})\frac{(e_0-e)^+}{ee_0}+\\
%&\frac{x-e}{e(e+e_0-x)e_0}]+\frac{e_0^2(e_0+e-x)}{(e_0+e+f-y'-x)(e+e_0-x+f-y')}\\
%&[\frac{(e_0-e)^+}{(e+e_0-x)ee_0}+\frac{3(f-e)^+}{(e+f-y')ef}+\frac{x-e_0}{e^2(e+e_0-x)}]\}\\
&\frac{x-e}{e(e+e_0-x)e_0}]+\frac{e_0^2(e_0+e-x)}{(e_0+e+f-y'-x)^2}[\frac{(e_0-e)^+}{(e+e_0-x)ee_0}+\frac{3(f-e)^+}{(e+f-y')ef}+\\
&\frac{x-e_0}{e^2(e+e_0-x)}]\}\\
=&(\frac{e_0^2}{e_0+e-x}+\frac{e_0^2}{e_0+e+f-y'-x})[(\frac{1}{e+e_0-x}+\frac{1}{e_0-d})\frac{(e_0-e)^+}{e}+\frac{x-e}{e(e+e_0-x)}]\\
&+\frac{(e_0+e-x)e_0^3}{(e_0+e+f-y'-x)^2}[\frac{(e_0-e)^+}{(e+e_0-x)ee_0}+\frac{3(f-e)^+}{(e+f-y')ef}+\frac{x-e_0}{e^2(e+e_0-x)}].
\end{align*}
Here is the new programming:
\begin{align*}
&\text{(P8): maximize } W^{(8)}_G \\
& \quad\text{s.t.}\\
%& \quad\text{s.t. for all } 1\leq i \leq R_0 \\
& \qquad\text{I. Degree constraints:} & x \in [1 - d, 1], \\
                                    & & y' \in [1 - d, 1], \\
& \qquad\text{II. Triangle constraints:} & e_0 \in [x - d, x], \\
                                       & & e \in [x + y' - 1, x], \\
                                       & & f \in [y' - d, y'].
\end{align*}

\begin{Lemma}\label{lem:P8>=P7}
%For $d<\frac{1}{7}$,  {\rm OPT(P8)$\geq$ OPT(P7)$\geq$ OPT(P1)}.
{\rm OPT(P8)$\geq$ OPT(P7)$\geq$ OPT(P1)}.
\end{Lemma}
\begin{proof}
Lemma \ref{lem:var>0_1} implies that $e_0,e,f,g_0$, $e+e_0-x>0$, $g_0+e-x>0$ and $e+f-y'>0$.

Let $g_1(s)=\frac{s^2}{s+e-x}$ and $g_2(s)=\frac{s^2}{s+e+f-y'-x}$. The derived functions of $g_1(s)$ and $g_2(s)$ are
$$g'_1(s)=\frac{s^2-2s(x-e)}{(s+e-x)^2}\ \ \text{and}\ \ g'_2(s)=\frac{s^2-2s(x+y'-e-f)}{(s+e+f-x-y')^2},$$
respectively. By the domain of (P7), $e_0-d+2(x-e)=e_0+2e-2x-d\geq x+2y'-2-2d\geq1-5d>0$ for $d<\frac{1}{5}$. Thus $g'_1(s)>0$ for any $s \in[e_0-d,e_0]$. That is, $g_1(s)$ is a monotone increasing function. So for any $s \in[e_0-d,e_0]$, $$g_1(s)\leq g_1(e_0).$$
Similarly, by the domain of (P7), $e_0-d-2(x+y'-e-f)=e_0+2e+2f-d-2x-2y'\geq x+2y'-2-4d\geq 1-7d>0$ for $d<\frac{1}{7}$. Therefore $g'_2(s)>0$ for any $s \in[e_0-d,e_0]$. That is, $g_2(s)$ is a monotone increasing function. So for any $s \in[e_0-d,e_0]$, $$g_2(s)\leq g_2(e_0).$$
Since $g_0+e-x>0$ by Lemma \ref{lem:var>0_1} and $g_0+e+f-y'-x\geq e_0+y'-1-2d\geq x+y'-1-3d\geq 1-5d>0$ because of $d<\frac{1}{5}$,
$$0<\frac{g_0^2}{g_0+e-x}+\frac{g_0^2}{g_0+e+f-y'-x}\leq\frac{e_0^2}{e_0+e-x}+\frac{e_0^2}{e_0+e+f-y'-x}.$$

Since $g_0 \in[e_0-d,e_0]$, $0<\frac{1}{g_0}\leq\frac{1}{e_0-d}$, and so $(\frac{1}{e+e_0-x}+\frac{1}{e_0-d})\frac{(e_0-e)^+}{ee_0}\geq (\frac{1}{e+e_0-x}+\frac{1}{g_0})\frac{(e_0-e)^+}{ee_0}\geq0$.

Claim that $$\frac{(g_0-e-e_0+x)^+}{e(e+e_0-x)g_0}\leq \frac{x-e}{e(e+e_0-x)e_0}.$$
Indeed, if $g_0\leq e+e_0-x$, then $\frac{(g_0-e-e_0+x)^+}{e(e+e_0-x)g_0}=0\leq \frac{x-e}{e(e+e_0-x)e_0}$. Suppose that $g_0>e+e_0-x$.
We know that $\frac{(g_0-e-e_0+x)^+}{e(e+e_0-x)g_0}=\frac{1}{e}(\frac{1}{e+e_0-x}-\frac{1}{g_0})\leq \frac{1}{e}(\frac{1}{e+e_0-x}-\frac{1}{e_0})=\frac{x-e}{e(e+e_0-x)e_0}.$

Note that $g_0+e+f-y'-x\geq e_0+y'-1-2d\geq x+y'-1-3d\geq 1-5d>0$ because of $d<\frac{1}{5}$. Since $g_0 \in[e_0-d,e_0]$, $g_0+e-x\leq e_0+e-x$ and $e_0+e-x+f-y'\geq g_0+e+f-y'-x>0$. By $g_2(s)\leq g_2(e_0)$ for any $s \in[e_0-d,e_0]$,
$$\frac{e_0^2(e_0+e-x)}{(e_0+e+f-y'-x)^2}\geq\frac{g_0^2(g_0+e-x)}{(g_0+e+f-y'-x)(e+e_0-x+f-y')}>0.$$
By $x\geq e_0$, Lemma \ref{lem:var>0_1} and the definition of ramp functions, $\frac{(e_0-e)^+}{(e+e_0-x)ee_0}+\frac{3(f-e)^+}{(e+f-y')ef}+\frac{x-e_0}{e^2(e+e_0-x)}\geq0$.

To sum up,  $W^{(8)}_G\geq W^{(7)}_G$. Because the constraints of (P8) are a subset of the constraints of (P7), OPT(P8) $\geq$ OPT(P7). By Lemma \ref{lem:P7>=P6}, OPT(P8) $\geq$ OPT(P7) $\geq$ OPT(P1).
\end{proof}

We now proceed with reducing $e$ as follows.

\begin{Lemma}\label{lem:P8_maximum}
The maximum of (P8) is achieved when $e=x+y'-1$ holds.
\end{Lemma}
\begin{proof}
Let $P_0=(x,y',e_0,e,f)$ be a point that achieves the maximum of (P8). Let $e'=x+y'-1$ and $P'_0=(x,y',e_0,e',f)$. Then we know that $P'_0$ is in the domain of (P8).

It suffices to prove $W^{(8)}_G(P'_0)\geq W^{(8)}_G(P_0)$ in the following.
It follows from  Lemma \ref{lem:var>0_1} that $e_0,f>0$ and $e\geq e'=x+y'-1>0$.

Since $e_0+e'-x=e_0+y'-1>0$ and $e_0+e'+f-y'-x=e_0+f-1>0$ by Lemma \ref{lem:N(S)_3},
%$$\frac{e_0^2}{e_0+e-x}+\frac{e_0^2}{e_0+e+f-y'-x}\leq \frac{e_0^2}{e_0+y'-1}+\frac{e_0^2}{e_0+f-1}\ \ \text{and }\ \ \frac{1}{e_0+e-x}\leq \frac{1}{e_0+y'-1}.$$
$$0<\frac{e_0^2}{e_0+e-x}+\frac{e_0^2}{e_0+e+f-y'-x}\leq \frac{e_0^2}{e_0+e'-x}+\frac{e_0^2}{e_0+e'+f-y'-x}.$$
Since $e_0-d\geq x-2d\geq1-3d>0$ for $d<\frac{1}{5}$, $$0<\frac{1}{e_0+e-x}+\frac{1}{e_0-d}\leq \frac{1}{e_0+e'-x}+\frac{1}{e_0-d}.$$

Claim that
\begin{align}\label{ineqn:e_1}
0\leq\frac{(e_0-e)^+}{e}\leq \frac{(e_0-e')^+}{e'}.
\end{align}
%$$\frac{(e_0-e)^+}{e}\leq \frac{(e_0-e')^+}{e'}.$$
Indeed, if $e_0\leq e$, then $\frac{(e_0-e)^+}{e}=0\leq \frac{(e_0-e')^+}{e'}$ since $e_0-e'=e_0-x-y'+1\geq 1-d-y'\geq0$. If $e_0>e$, then $0<\frac{(e_0-e)^+}{e}=\frac{e_0}{e}-1\leq \frac{e_0}{e'}-1=\frac{(e_0-e')^+}{e'}$ by $e\geq e'$.

By $e\geq e'$, we have $(x-e)e_0\leq(x-e')e_0$, and so $\frac{x-e}{e+e_0-x}\leq \frac{x-e'}{e'+e_0-x}$. Since $e\in[x+y'-1,x]$, $\frac{x-e}{e+e_0-x}\geq0$. Then
$$0\leq\frac{x-e}{e(e+e_0-x)}\leq \frac{x-e'}{e'(e'+e_0-x)}.$$
Since $y'\geq f$ and $e\geq e'$, $(f-y')(e_0+e-x)\leq(f-y')(e_0+e'-x)$. So we know $0<\frac{e_0+e-x}{e_0+e+f-y'-x}\leq \frac{e_0+e'-x}{e_0+e'+f-y'-x}$. Therefore,
$$0<\frac{(e_0+e-x)e_0^3}{(e_0+e+f-y'-x)^2}\leq \frac{(e_0+e'-x)e_0^3}{(e_0+e'+f-y'-x)^2}.$$
By (\ref{ineqn:e_1}) and $e\geq e'$, $0\leq\frac{(e_0-e)^+}{(e+e_0-x)ee_0}\leq \frac{(e_0-e')^+}{(e'+e_0-x)e'e_0}$.

Claim that $$0\leq\frac{3(f-e)^+}{(e+f-y')ef}\leq \frac{3(f-e)^+}{(e'+f-y')e'f}.$$
Indeed, since $e'+f-y'=x+f-1\geq 1-3d>0$ for $d<\frac{1}{5}$, $\frac{3}{(e+f-y')}\leq \frac{3}{(e'+f-y')}$. If $f\leq e$, then $\frac{(f-e)^+}{ef}=0\leq \frac{(f-e')^+}{e'f}$. If $f>e$, then $\frac{(f-e)^+}{ef}=\frac{1}{e}-\frac{1}{f}\leq\frac{1}{e'}-\frac{1}{f}=\frac{(f-e')^+}{e'f}$. Therefore, $\frac{3(f-e)^+}{(e+f-y')ef}\leq \frac{3(f-e)^+}{(e'+f-y')e'f}.$

Since $e\geq e'$ and $x\geq e_0$, $0\leq\frac{x-e_0}{e^2(e+e_0-x)}\leq \frac{x-e_0}{e'^2(e'+e_0-x)}$. To sum up, $W^{(8)}_G(P'_0)\geq W^{(8)}_G(P_0)$.
\end{proof}

Thus we may replace (P8) with a new programming (P9) whose maximum value is at least that of (P8) in the following.
Let
\begin{align*}
\widehat{W}^{(9)}_G
=(&\frac{e_0^2}{e_0+y'-1}+\frac{e_0^2}{e_0+f-1})[(\frac{1}{e_0+y'-1}+\frac{1}{e_0-d})\frac{(e_0-x-y'+1)^+}{x+y'-1}+\\
&\frac{1-y'}{(x+y'-1)(e_0+y'-1)}]+\frac{(e_0+y'-1)e_0^3}{(e_0+f-1)^2}[\frac{(e_0-x-y'+1)^+}{(e_0+y'-1)(x+y'-1)e_0}+\\
&\frac{3(f-x-y'+1)^+}{(x+f-1)(x+y'-1)f}+\frac{x-e_0}{(x+y'-1)^2(e_0+y'-1)}].
\end{align*}
Here is the new programming:
\begin{align*}
&\text{(P9): maximize } \widehat{W}^{(9)}_G \\
& \quad\text{s.t.}\\
%& \quad\text{s.t. for all } 1\leq i \leq R_0 \\
& \qquad\text{I. Degree constraints:} & x \in [1 - d, 1], \\
                                    & & y' \in [1 - d, 1], \\
& \qquad\text{II. Triangle constraints:} & e_0 \in [x - d, x], \\
                                       & & f \in [y' - d, y'].
\end{align*}
\begin{Corollary}\label{cor:P9>=P8}
 {\rm OPT(P1)$\leq$ OPT(P8)= OPT(P9)}.
\end{Corollary}
\begin{proof}
%By Lemmas \ref{lem:P8_maximum} and \ref{lem:P8>=P7}, the desired result is obtained.
Combining Lemmas \ref{lem:P8_maximum} and \ref{lem:P8>=P7}, we complete the proof.
\end{proof}

\subsection{Proof of Theorem \ref{thm:W'_G(O)<=1}}\label{sec:proof_3.5}

Before proceeding, it will be useful to switch the variables. To that end we introduce two new variables $a$, $b$ to replace $f$, $e_0$ respectively as follows:
\begin{itemize}
  \item $e_0 = x-a$,
  \item $f = y'-b$.
\end{itemize}
We rewrite the objective function $\widehat{W}^{(9)}_G$ of the programming (P9).
\begin{align*}
W^{(9)}_G
&=[\frac{(x-a)^2}{x-a+y'-1}+\frac{(x-a)^2}{x-a+y'-b-1}][(\frac{1}{x-a+y'-1}+\frac{1}{x-a-d})\frac{(1-y'-a)^+}{x+y'-1}+\\
&\frac{1-y'}{(x+y'-1)(x-a+y'-1)}]+\frac{(x-a+y'-1)(x-a)^3}{(x+y'-a-b-1)^2}[\frac{(1-y'-a)^+}{(x-a+y'-1)(x+y'-1)(x-a)}\\
&+\frac{3(1-x-b)^+}{(x+y'-b-1)(x+y'-1)(y'-b)}+\frac{a}{(x+y'-1)^2(x-a+y'-1)}]\\
&=[\frac{(x-a)^2}{x-a+y'-1}+\frac{(x-a)^2}{x-a+y'-b-1}][\frac{(1-y'-a)^+ +1-y'}{(x-a+y'-1)(x+y'-1)}+\frac{(1-y'-a)^+}{(x-a-d)(x+y'-1)}]\\
&+\frac{(x-a)^2(1-y'-a)^+}{(x+y'-1)(x+y'-a-b-1)^2}+\frac{a(x-a)^3}{(x+y'-a-b-1)^2(x+y'-1)^2}\\
&+\frac{3(x-a+y'-1)(x-a)^3(1-x-b)^+}{(x+y'-b-1)(x+y'-1)(y'-b)(x+y'-a-b-1)^2}].
\end{align*}
Here then is programming (P9) with these new variables and constraints:
\begin{align*}
&\text{(P9): maximize } W^{(9)}_G \\
& \quad\text{s.t.}\\
%& \quad\text{s.t. for all } 1\leq i \leq R_0 \\
& \qquad\text{I. Degree constraints:} & x \in [1 - d, 1], \\
                                    & & y' \in [1 - d, 1], \\
& \qquad\text{II. Triangle constraints:} & a \in [0,d], \\
                                       & & b \in [0,d].
\end{align*}

\begin{Lemma}\label{lem:P9_maximum}
The maximum of (P9) is achieved when $y'=1-d$ holds.
\end{Lemma}
\begin{proof}
Let $P_0=(x,y',a,b)$ be a point that achieves the maximum of (P9). Let $y''=1-d$ and $P'_0=(x,y'',a,b)$. Then we know that $P'_0$ is in the domain of (P9).

It suffices to prove $W^{(9)}_G(P'_0)\geq W^{(9)}_G(P_0)$ in the following. Since $P_0$ is in the domain of (P9), $x+y''-a-b-1=x-a-b-d\geq1-4d>0$ for $d<\frac{1}{5}$, and so
\begin{align}\label{ineqn:y_1}
0<\frac{1}{x-a+y'-b-1}\leq\frac{1}{x-a+y''-b-1}.
\end{align}
Note that $x-y''-1\geq x+y''-a-1\geq x+y''-a-b-1>0$. So we have
\begin{align}\label{ineqn:y_2}
0<\frac{1}{x-a+y'-1}\leq \frac{1}{x-a+y''-1}\ \ \text{and}\ \ 0<\frac{1}{x+y'-1}\leq\frac{1}{x+y''-1}.
\end{align}
Since $x+y''-b-1\geq x+y''-a-b-1>0$,
\begin{align}\label{ineqn:y_3}
0<\frac{1}{x+y'-b-1}\leq\frac{1}{x+y''-b-1}.
\end{align}
Similarly, $y''-b\geq1-2d>0$ for $d<\frac{1}{5}$. Then
\begin{align}\label{ineqn:y_4}
0<\frac{1}{y'-b}\leq\frac{1}{y''-b}.
\end{align}

Claim that
\begin{align}\label{ineqn:y_5}
0<\frac{(1-y'-a)^+}{x+y'-1}\leq \frac{(1-y''-a)^+}{x+y''-1}.
\end{align}
Indeed, if $y'\geq 1-a$, then $\frac{(1-y'-a)^+}{x+y'-1}=0\leq\frac{(1-y''-a)^+}{x+y''-1}$. Suppose that $y'<1-a$ in the following. For $d<\frac{1}{5}$, $x-a\geq1-2d>0$, and so $\frac{x-a}{x+y'-1}\leq\frac{x-a}{x+y''-1}$. Then $\frac{(1-y'-a)^+}{x+y'-1}=\frac{1-y'-a}{x+y'-1}=\frac{x-a}{x+y'-1}-1\leq\frac{x-a}{x+y''-1}-1=\frac{1-y''-a}{x+y''-1}=\frac{(1-y''-a)^+}{x+y''-1}$.

By (\ref{ineqn:y_1}) and (\ref{ineqn:y_2}), $0\leq\frac{(x-a)^2}{x-a+y'-1}+\frac{(x-a)^2}{x-a+y'-b-1}\leq\frac{(x-a)^2}{x-a+y''-1}+\frac{(x-a)^2}{x-a+y''-b-1}.$
By (\ref{ineqn:y_2}) and (\ref{ineqn:y_5}), $0<\frac{(1-y'-a)^+}{(x-a+y'-1)(x+y'-1)}\leq\frac{(1-y''-a)^+}{(x-a+y''-1)(x+y''-1)}.$
Since $1-y'=d\geq0$, $0\leq\frac{1-y'}{x+y'-1}=\frac{x}{x+y'-1}-1\leq \frac{x}{x+y''-1}-1=\frac{1-y''}{x+y''-1}$.
It follows from (\ref{ineqn:y_2}) that $0<\frac{1-y'}{(x-a+y'-1)(x+y'-1)}\leq\frac{1-y''}{(x-a+y''-1)(x+y''-1)}.$
Therefore, $0<\frac{(1-y'-a)^+ +1-y'}{(x-a+y'-1)(x+y'-1)}\leq\frac{(1-y''-a)^+ +1-y''}{(x-a+y''-1)(x+y''-1)}.$

Since $x-a-d\geq x-a-b-d>0$, $0\leq\frac{(1-y'-a)^+}{(x-a-d)(x+y'-1)}\leq\frac{(1-y''-a)^+}{(x-a-d)(x+y''-1)}$ by (\ref{ineqn:y_5}).
It follows from (\ref{ineqn:y_1}) and (\ref{ineqn:y_5}) that $0\leq\frac{(x-a)^2(1-y'-a)^+}{(x+y'-1)(x+y'-a-b-1)^2}\leq\frac{(x-a)^2(1-y''-a)^+}{(x+y''-1)(x+y''-a-b-1)^2}$.
Since $a\geq0$ and $x-a\geq1-2d>0$ for $d<\frac{1}{2}$, $0\leq\frac{a(x-a)^3}{(x+y'-a-b-1)^2(x+y'-1)^2}\leq\frac{a(x-a)^3}{(x+y''-a-b-1)^2(x+y''-1)^2}$ by (\ref{ineqn:y_1}) and (\ref{ineqn:y_2}).
Since $b\geq0$, $-b(x-a+y'-1)\leq-b(x-a+y''-1)$ and so $\frac{x-a+y'-1}{x-a+y'-b-1}\leq\frac{x-a+y''-1}{x-a+y''-b-1}.$
By (\ref{ineqn:y_1}), (\ref{ineqn:y_2}), (\ref{ineqn:y_3}) and (\ref{ineqn:y_4}), $0\leq\frac{1}{(x+y'-b-1)(x+y'-1)(y'-b)(x+y'-a-b-1)^2}\leq\frac{1}{(x+y''-b-1)(x+y''-1)(y''-b)(x+y''-a-b-1)^2}$,
and so $0\leq\frac{3(x-a+y'-1)(x-a)^3(1-x-b)^+}{(x+y'-b-1)(x+y'-1)(y'-b)(x+y'-a-b-1)^2}\leq\frac{3(x-a+y''-1)(x-a)^3(1-x-b)^+}{(x+y''-b-1)(x+y''-1)(y''-b)(x+y''-a-b-1)^2}$.

Therefore, $W^{(9)}_G(P'_0)\geq W^{(9)}_G(P_0)$. Then the desired conclusion is obtained.
\end{proof}

Thus we construct a new programming (P10) with the following new objective function and a subset of the previous constraints as follows:

\begin{align*}
W^{(10)}_G
=&[\frac{(x-a)^2}{x-a-d}+\frac{(x-a)^2}{x-a-d-b}]\frac{3d-2a}{(x-a-d)(x-d)}+\frac{(x-a)^2(d-a)}{(x-d)(x-d-a-b)^2}+\\
&\frac{a(x-a)^3}{(x-a-b-d)^2(x-d)^2}+\frac{3(x-a-d)(x-a)^3(1-x-b)^+}{(x-b-d)(x-d)(1-d-b)(x-a-b-d)^2}].
\end{align*}
Here is the new programming:
\begin{align*}
&\text{(P10): maximize } W^{(10)}_G \\
& \quad\text{s.t.}\\
%& \quad\text{s.t. for all } 1\leq i \leq R_0 \\
& \qquad\text{I. Degree constraints:} & x \in [1 - d, 1], \\
& \qquad\text{II. Triangle constraints:} & a \in [0,d], \\
                                       & & b \in [0,d].
\end{align*}

\begin{Corollary}\label{cor:P10>=P9}
 {\rm OPT(P1)$\leq$ OPT(P9)= OPT(P10)}.
\end{Corollary}
\begin{proof}
By Corollary \ref{cor:P9>=P8} and Lemma \ref{lem:P9_maximum}, the desired conclusion holds.
\end{proof}

\begin{Lemma}\label{lem:P10_maximum}
The maximum of (P10) is achieved when $x=1-d$ holds.
\end{Lemma}
\begin{proof}
Let $P_0=(x,a,b)$ be a point that achieves the maximum of (P10). Let $x'=1-d$ and $P'_0=(x',a,b)$. Then we know that $P'_0$ is in the domain of (P10).

It suffices to prove $W^{(10)}_G(P'_0)\geq W^{(10)}_G(P_0)$. Note that $x\geq x'$. Since $x'-a-b-d\geq1-4d>0$ for $d<\frac{1}{5}$, $x'-d\geq x'-b-d\geq x'-a-b-d>0$ and $x'-a\geq x'-a-d\geq x'-a-b-d>0$.
Since $-d(x-a)\leq -d(x'-a)$,
\begin{align}\label{ineqn:x_1}
0<\frac{x-a}{x-a-d}\leq\frac{x'-a}{x'-a-d}.
\end{align}
Since $a\leq d$, $(a-d)(x-a)\leq (a-d)(x'-a)$, and so
\begin{align}\label{ineqn:x_2}
0<\frac{x-a}{x-d}\leq\frac{x'-a}{x'-d}.
\end{align}
By $-(b+d)(x-a)\leq -(b+d)(x'-a)$,
\begin{align}\label{ineqn:x_3}
0<\frac{x-a}{x-a-b-d}\leq\frac{x'-a}{x'-a-b-d}.
\end{align}
By $-b(x-a-d)\leq -b(x'-a-d)$,
\begin{align}\label{ineqn:x_4}
0<\frac{x-a-d}{x-a-b-d}\leq\frac{x'-a-d}{x'-a-b-d}.
\end{align}
Since $a\in[0, d]$, $(a-b-d)(x-a)\leq (a-b-d)(x'-a)$, and so
\begin{align}
\label{ineqn:x_5}
0<\frac{x-a}{x-b-d}\leq\frac{x'-a}{x'-b-d}.
\end{align}

Since $a\in[0,d]$, $3d-2a\geq0$ and so $0\leq\frac{(3d-2a)(x-a)^2}{(x-a-d)(x-d)}\leq\frac{(3d-2a)(x-a)^2}{(x'-a-d)(x'-d)}$ by (\ref{ineqn:x_1}) and (\ref{ineqn:x_2}). Recall $x'-a-d\geq x'-a-d-b>0$. So $0\leq\frac{1}{x-a-d}+\frac{1}{x-a-d-b}\leq\frac{1}{x'-a-d}+\frac{1}{x'-a-d-b}$. Thus $0\leq[\frac{(x-a)^2}{x-a-d}+\frac{(x-a)^2}{x-a-d-b}]\frac{3d-2a}{(x-a-d)(x-d)}\leq[\frac{(x'-a)^2}{x'-a-d}+\frac{(x'-a)^2}{x'-a-d-b}]\frac{3d-2a}
{(x'-a-d)(x'-d)}$.

By $x-d\geq x'-d>0$ and $a\in[0, d]$, $0\leq \frac{d-a}{x-d}\leq \frac{d-a}{x'-d}$. It follows from (\ref{ineqn:x_3}) that $0\leq \frac{(x-a)^2(d-a)}{(x-d)(x-d-a-b)^2}\leq \frac{(x'-a)^2(d-a)}{(x'-d)(x'-d-a-b)^2}$.
By (\ref{ineqn:x_2}) and (\ref{ineqn:x_3}), $0\leq\frac{a(x-a)^3}{(x-a-b-d)^2(x-d)^2}\leq\frac{a(x'-a)^3}{(x'-a-b-d)^2(x'-d)^2}$.

By $x\geq x'$, $\frac{1}{x-a-b-d}\leq\frac{1}{x'-a-b-d}$ and $1-x-b\leq 1-x'-b$. By the definition of a ramp function, $(1-x-b)^+\leq (1-x'-b)^+$.
Since $1-d-b\geq1-2d>0$ for $d<\frac{1}{2}$, $0\leq\frac{3(x-a-d)(x-a)^3(1-x-b)^+}{(x-b-d)(x-d)(1-d-b)(x-a-b-d)^2}\leq\frac{3(x'-a-d)(x'-a)^3(1-x'-b)^+}{(x'-b-d)(x'-d)(1-d-b)(x'-a-b-d)^2}$ by (\ref{ineqn:x_3}), (\ref{ineqn:x_4}) and (\ref{ineqn:x_5}).

To sum up, we obtain $W^{(10)}_G(P'_0)\geq W^{(10)}_G(P_0)$. So $P'_0$ is a point with $x=1-d$ that achieves the maximum of (P10).
\end{proof}

Thus we may replace (P10) with a new programming (P11) whose maximum value is at least that of (P10) by setting $x = 1 - d$ as follows:

\begin{align*}
W^{(11)}_G
=&[\frac{(1-d-a)^2}{1-a-2d}+\frac{(1-d-a)^2}{1-a-2d-b}]\frac{3d-2a}{(1-a-2d)(1-2d)}+\frac{(1-d-a)^2(d-a)}{(1-2d)(1-2d-a-b)^2}+\\
&\frac{a(1-d-a)^3}{(1-a-b-2d)^2(1-2d)^2}+\frac{3(1-a-2d)(1-d-a)^3(d-b)}{(1-b-2d)(1-2d)(1-d-b)(1-a-b-2d)^2}].
%=&[\frac{(1-d-a)^2}{1-a-2d}+\frac{(1-d-a)^2}{1-a-2d-b}]\frac{3d-2a}{(1-a-2d)(1-2d)}+\frac{(1-d-a)^2(d-a)}{(1-2d)(1-2d-a-b)^2}+\\
%&\frac{a(1-d-a)^3}{(1-a-b-2d)^2(1-2d)^2}+\frac{3(1-a-2d)(1-d-a)^3}{(1-b-2d)(1-2d)(1-a-b-2d)^2}].
\end{align*}
Here is the new programming:
\begin{align*}
&\text{(P11): maximize } W^{(11)}_G \\
& \quad\text{s.t.} \ \ a,b \in [0,d].
\end{align*}
\begin{Corollary}\label{cor:P11>=P10}
 {\rm OPT(P1)$\leq$ OPT(P10)= OPT(P11)}.
\end{Corollary}
\begin{proof}
It follows from Corollary \ref{cor:P10>=P9} and Lemma \ref{lem:P10_maximum} that the desired result is obtained.
\end{proof}

We construct a new programming (P12) with the following new objective function.

\begin{align*}
W^{(12)}_G
=&[\frac{(1-d-a)^2}{1-a-2d}+\frac{(1-d-a)^2}{1-a-3d}]\frac{3d-2a}{(1-a-2d)(1-2d)}+\frac{(1-d-a)^2(d-a)}{(1-2d)(1-3d-a)^2}+\\
&\frac{a(1-d-a)^3}{(1-a-3d)^2(1-2d)^2}+\frac{3(1-a-2d)(1-d-a)^3d}{(1-3d)(1-2d)^2(1-a-3d)^2}].
\end{align*}
Here is the new programming:
\begin{align*}
&\text{(P12): maximize } W^{(12)}_G \\
& \quad\text{s.t.} \ \ a \in [0,d].
\end{align*}

\begin{Lemma}\label{lem:P12>=P11}
{\rm OPT(P1)$\leq$ OPT(P11)$\leq$ OPT(P12)}.
%The maximum of (P11) is achieved when $b=d$ holds.
\end{Lemma}
\begin{proof}
Since $0\leq a,b\leq d$, $1-a-2d-b\geq1-a-3d\geq 1-4d>0$ for $d<\frac{1}{5}$, and so
\begin{align}\label{ineqn:b_1}
0<\frac{1}{1-a-2d-b}\leq\frac{1}{1-a-3d}.
\end{align}
%$0<\frac{(1-d-a)^2}{1-a-2d-b}\leq\frac{(1-d-a)^2}{1-a-3d}$.
By $3d-2a\geq0$ and $1-2d\geq 1-a-2d\geq1-a-2d-b>0$, $\frac{3d-2a}{(1-a-2d)(1-2d)}>0$, and so
$$[\frac{(1-d-a)^2}{1-a-2d}+\frac{(1-d-a)^2}{1-a-2d-b}]\frac{3d-2a}{(1-a-2d)(1-2d)}\leq
[\frac{(1-d-a)^2}{1-a-2d}+\frac{(1-d-a)^2}{1-a-3d}]\frac{3d-2a}{(1-a-2d)(1-2d)}.$$
By (\ref{ineqn:b_1}) and $d-a\geq0$, $$0\leq\frac{(1-d-a)^2(d-a)}{(1-2d)(1-2d-a-b)^2}\leq\frac{(1-d-a)^2(d-a)}{(1-2d)(1-3d-a)^2}.$$
Since $a\geq 0$ and $1-d-a\geq 1-2d>0$, by (\ref{ineqn:b_1}), $$0\leq\frac{a(1-d-a)^3}{(1-a-b-2d)^2(1-2d)^2}\leq\frac{a(1-d-a)^3}{(1-a-3d)^2(1-2d)^2}.$$
Since $0\leq a\leq d$, $1-b-2d\geq1-a-2d-b>0$ and so $0<\frac{1}{1-b-2d}\leq \frac{1}{1-3d}$. By $0\leq b\leq d$, $1-d-b\geq1-2d>0$, and so $0<\frac{1}{1-b-d}\leq \frac{1}{1-2d}$. It follows from $b\geq0$ and (\ref{ineqn:b_1}) that
$$0\leq\frac{3(1-a-2d)(1-d-a)^3(d-b)}{(1-b-2d)(1-2d)(1-d-b)(1-a-b-2d)^2}\leq\frac{3(1-a-2d)(1-d-a)^3d}{(1-3d)(1-2d)^2(1-a-3d)^2}.$$

Thus, $W^{(12)}_G\geq W^{(11)}_G$. So we complete the proof by Corollary \ref{cor:P11>=P10}.
\end{proof}

Let
\begin{align*}
W^{(13)}_G
%=&[\frac{(1-d)^2}{1-2d}+\frac{(1-d)^2}{1-3d}]\frac{3d}{(1-2d)^2}+\frac{(1-d)^2d}{(1-2d)(1-3d)^2}+\frac{d(1-d)^3}{(1-3d)^2(1-2d)^2}\\
%&+\frac{3(1-2d)(1-d)^3}{(1-3d)(1-2d)(1-3d)^2}]\\
=&[\frac{(1-d)^2}{1-2d}+\frac{(1-d)^2}{1-3d}]\frac{3d}{(1-2d)^2}+\frac{(1-d)^2d}{(1-2d)(1-3d)^2}+\frac{d(1-d)^3}{(1-3d)^2(1-2d)^2}\\
&+\frac{3(1-d)^3d}{(1-3d)^3(1-2d)}].
\end{align*}
%\begin{align*}
%W^{(13)}_G
%=&[\frac{(1-d)^2}{1-2d}+\frac{(1-d)^2}{1-3d}]\frac{3d}{(1-2d)^2}+\frac{(1-d)^2d}{(1-2d)(1-3d)^2}+\frac{d(1-2d)(1-d)^2}{(1-3d)(1-4d)(1-2d)^2}\\
%&+\frac{3(1-2d)(1-d)^3}{(1-3d)(1-2d)(1-3d)^2}]\\
%=&[\frac{(1-d)^2}{1-2d}+\frac{(1-d)^2}{1-3d}]\frac{3d}{(1-2d)^2}+\frac{(1-d)^2d}{(1-2d)(1-3d)^2}+\frac{d(1-d)^2}{(1-3d)(1-4d)(1-2d)}\\
%&+\frac{3(1-d)^3}{(1-3d)^3}].
%\end{align*}
\begin{Lemma}\label{lem:P13>=P12}
{\rm OPT(P1)$\leq$ OPT(P12)$\leq W^{(13)}_G$ }.
%For $d<\frac{1}{8}$, {\rm OPT(P1)$\leq$ OPT(P12)$\leq W^{(13)}_G$ }.
\end{Lemma}
\begin{proof}
Let $g_1(a)=\frac{(1-d-a)^2}{1-a-2d}$ and $g_2(a)=\frac{(1-d-a)^2}{1-a-3d}$. Their derived functions are
$$g'_1(a)=\frac{(1-d-a)(a+3d-1)}{(1-a-2d)^2}\ \text{ and }\ g'_2(a)=\frac{(1-d-a)(a+5d-1)}{(1-a-3d)^2}.$$
Since $0\leq d<\frac{1}{6}$, $d<1-5d\leq 1-3d$, and so $g'_1(a)<0$ and $g'_2(a)<0$ for $a\in[0,d]$. That is, $g_1(a)$ and $g_2(a)$ are two decreasing functions for $a\in[0,d]$. Thus for $a\in [0,d]$, $$g_1(a)\leq g_1(0)\ \text{ and }\ g_2(a)\leq g_2(0).$$
Since $1-a-2d\geq1-a-3d\geq1-4d>0$ for $d<\frac{1}{4}$, $$0\leq\frac{(1-d-a)^2}{1-a-2d}+\frac{(1-d-a)^2}{1-a-3d}\leq \frac{(1-d)^2}{1-2d}+\frac{(1-d)^2}{1-3d}.$$

By $d<\frac{2}{7}$ and $a\geq 0$, $-2a(1-2d)< -3da$, and so $(3d-2a)(1-2d)\leq 3d(1-2d-a)$. Since $a\leq d$ and $d<\frac{1}{3}$, $1-2d\geq1-a-2d\geq1-3d>0$. Thus $$0<\frac{3d-2a}{1-2d-a}\leq \frac{3d}{1-2d}.$$

Since $d<\frac{1}{4}$ and $a\geq0$, $-a(1-3d)<-ad$, and so $(d-a)(1-3d)\leq d(1-3d-a)$. By $a\leq d$, $1-3d\geq 1-3d-a>0$. Thus
$0\leq\frac{d-a}{1-3d-a}\leq \frac{d}{1-3d}$. Recall that $g_2(a)\leq g_1(0)$ and $1-2d\geq 1-3d>0$. So
$$0\leq\frac{(1-d-a)^2(d-a)}{(1-2d)(1-3d-a)^2}\leq \frac{(1-d)^2d}{(1-2d)(1-3d)^2}.$$

Let $g_3(a)=\frac{(1-d-a)^3}{(1-a-3d)^2}$. Its derived function is $$g'_3(a)=\frac{(1-d-a)^2(1-a-3d)(a+7d-1)}{(1-a-3d)^4}.$$
Since $d<\frac{1}{8}$, $g'_3(a)<0$ for $a\in[0,d]$. That is, $g_3(a)$ is a decreasing function for $a\in[0,d]$. So $g_3(a)\leq g_3(0)$. By $a\leq d$ and $1-2d>0$,
$$\frac{a(1-d-a)^3}{(1-a-3d)^2(1-2d)^2}\leq \frac{d(1-d)^3}{(1-3d)^2(1-2d)^2}.$$

Since $a\geq0$, $1-a-2d\leq 1-2d$. By $d<\frac{1}{3}$, $1-3d>0$ and $1-2d>0$. Recall that $g_3(a)\leq g_3(0)$, and so $$\frac{3(1-a-2d)(1-d-a)^3d}{(1-3d)(1-2d)^2(1-a-3d)^2}\leq \frac{3(1-2d)(1-d)^3d}{(1-3d)(1-2d)^2(1-3d)^2}=\frac{3(1-d)^3d}{(1-3d)^3(1-2d)}.$$

To sum up, $W_{G}^{(13)}\geq W_{G}^{(12)}$, and so  $W_{G}^{(13)}\geq $ OPT(12). By Lemma \ref{lem:P12>=P11}, we complete the proof.
\end{proof}

\begin{proof}[{\bf Proof of Theorem \ref{thm:W'_G(O)<=1}}]
Lemmas \ref{lem:program->thm2.5} and \ref{lem:P13>=P12} imply that if $W^{(13)}_G\leq 1$, then the desired conclusion is obtained.
To prove $W^{(13)}_G\leq 1$, It suffices to show $W(d)<0$, where
\begin{align*}
W(d)=&3d(1-d)^2(1-3d)^3+3d(1-d)^2(1-3d)^2(1-2d)+(1-d)^2d(1-2d)^2(1-3d)\\
&+d(1-d)^3(1-3d)(1-2d)+3(1-d)^3d(1-2d)^2-(1-3d)^3(1-2d)^3\\
=&-1+26d-194d^2+669d^3-1192d^4+1065d^5-381d^6.
%=&[\frac{(1-d)^2}{1-2d}+\frac{(1-d)^2}{1-3d}]\frac{3d}{(1-2d)^2}+\frac{(1-d)^2d}{(1-2d)(1-3d)^2}+\frac{d(1-d)^3}{(1-3d)^2(1-2d)^2}\\
%&+\frac{3(1-d)^3d}{(1-3d)^3(1-2d)}].
\end{align*}
The first derivative, second derivative, third derivative and fourth derivative are
$$W'(d)=26-388d+2007d^2-4768 d^3 + 5325 d^4 - 2286 d^5,$$
$$W''(d)=-388 + 4014 d - 14304 d^2 + 21300 d^3 - 11430 d^4,$$
$$W'''(d)=4014 - 28608 d + 63900 d^2 - 45720 d^3$$
and
$$W^{(4)}(d)=-28606 + 127800 d - 137160 d^2,$$
respectively. By $\frac{-127800}{2\times(-137160)}>\frac{2}{33}$, $W^{(4)}(d)$ is an increasing function on $[0,\frac{2}{33}]$.
Since $W^{(4)}(\frac{2}{33})=\frac{-2585086}{121}<0$, $$W^{(4)}(d)<0$$ for any $d\in [0,\frac{2}{33}]$. That is, $W'''(d)$ is decreasing on $[0,\frac{2}{33}]$.
Since $W'''(\frac{2}{33})=\frac{10001326}{3993}>0$, $$W'''(d)>0$$ for any $d\in [0,\frac{2}{33}]$. We have $W''(d)$ is increasing on $[0,\frac{2}{33}]$.
It follows from $W''(\frac{2}{33})=-\frac{25389224}{131769}$ that $$W''(d)<0$$ for any $d\in [0,\frac{2}{33}]$. So $W'(d)$ is a decreasing function on $[0,\frac{2}{33}]$.
By $W'(\frac{2}{33})=\frac{38549710}{4348377}>0$, $$W'(d)>0$$ for any $d\in [0,\frac{2}{33}]$. We have $W(d)$ is an increasing function on $[0,\frac{2}{33}]$.
Note that $W(\frac{2}{33})=-\frac{1345519}{430489323}<0$. Thus for $d\in [0,\frac{2}{33}]$, $$W(d)<0.$$

\end{proof}
\section{Concluding remarks}

This paper is devoted to examining the fractional $K_4$-decomposition and $K_4$-decomposition of dense graphs. Theorem \ref{thm:frac_K_4_decom} implies that any graph on $n$ vertices with minimum degree at least $\frac{31}{33}n$ has a fractional $K_4$-decomposition. Combining the result of Glock, K\"{u}hn, Lo, Montgomery and Osthus, all large enough $K_4$-divisible graphs on $n$ vertices with minimum degree at least $(\frac{31}{33}+\varepsilon)n$ admit a $K_4$-decomposition.

%Conjecture \ref{conj:Nash-Williams_conj} conjectures that every large enough $K_4$-divisible graph on $n$ vertices with minimum degree at least $\frac{4}{5}n$ admits a $K_4$-decomposition.
%Thus an important task is to improve the upper bound of $\delta^*_{K_4}$ and even determine the correct value of $\delta^*_{K_4}$.
For the $K_4$-decomposition of dense graphs, an important task is to improve the upper bound of $\delta^*_{K_4}$.
To obtain Theorem \ref{thm:frac_K_4_decom}, we solve a nonlinear programming by slowly reducing the number of variables in Section \ref{sec:solve_program}. We construct some new programmings such that each of their optimum values is strictly greater than that of the original programming in this process. It is natural to wonder if the value of $d$ could be improved by solving the original programming using a different method.

Finally, this paper only investigates the case of $K_4$ subject to the limitations of using the nonlinear programming. A natural question is how to reduce the upper bound of $\delta^*_{K_r}$ for $r\geq5$.

\end{document}